\documentclass[11pt]{amsart}
\usepackage{geometry}                
\usepackage{graphicx}
\usepackage{amssymb}
\usepackage{pdfsync}
\usepackage{hyperref}
\usepackage{epstopdf}
\def\R{\mathbb{R}}

\def\N{\mathbb{N}}

\def\Co{\mathbb{C}}

\def\C{\mathcal{C}}

\def\P{\mathbb{P}}
\def\H{\mathcal{H}}
\def\F{\mathcal{F}}

\newtheorem{theorem}{Theorem}
\newtheorem{corollary}[theorem]{Corollary}
\newtheorem{definition}[theorem]{Definition}
\newtheorem{remark}[theorem]{Remark}
\newtheorem{proposition}[theorem]{Proposition}
\newtheorem{lemma}[theorem]{Lemma}

\newtheorem{claim}[theorem]{Claim}

\title{Extremizers for Fourier restriction inequalities: convex arcs}
\author{Diogo Oliveira e Silva}
\address{
        Diogo Oliveira e Silva\\
        Department of Mathematics\\
        University of California \\
        Berkeley, CA 94720-3840, USA}
\email{dosilva@math.berkeley.edu}
\thanks{The  author was partially supported by the Funda\c{c}\~{a}o para a Ci\^{e}ncia e a Tecnologia 
(FCT/Portugal grant SFRH/BD/28041/2006) 
and the National Science Foundation under agreement DMS-0901569.}

\date{\today}                                           

\begin{document}
\maketitle

\begin{abstract}
{\noindent We establish the existence of extremizers for a Fourier restriction inequality on planar convex arcs without points with colinear tangents whose curvature satisfies a natural assumption. More generally, we prove that any extremizing sequence of nonnegative functions has a subsequence which converges to an extremizer.}
\end{abstract}

\tableofcontents

\section{Introduction}
Consider a compact arc  $\Gamma\subset\R^2$ of a smooth, convex curve equipped with arclength measure $\sigma$. Assume the curvature $\kappa$ of $\Gamma$ to be positive everywhere; equivalently, assume that $\lambda:=\min_\Gamma\kappa$ is  a positive real number. 
Let $\ell:=\sigma(\Gamma)$ and parametrize $\Gamma$ by arclength:

\begin{displaymath}
\begin{array}{rcl}
\gamma: [0,\ell] & \rightarrow & \R^2\\
s & \mapsto & \gamma(s)=(x(s),y(s)).
\end{array}
\end{displaymath}

For $s\in [0,\ell]$, let $t(s)=(x'(s),y'(s))$ be the tangent indicatrix and let $\theta(s)\in S^1$ measure the net rotation described by the vector $t(s)$ as we run the curve $\gamma$ from $0$ to $s$. In other words, if we let $e_1=(1,0)$, then $\theta$ is the unique continuous function satisfying $t(s)=(\cos \theta(s),\sin \theta(s))$ for every $s\in [0,\ell]$, and such that $\theta(0)=\textrm{arc}\cos( e_1\cdot t(0) )$. The function $\theta$ is related to the curvature $\kappa$ via  
$$\theta(s)=\int_0^s \kappa(t)dt.$$

We further assume that the arc $\Gamma$ has no points with colinear tangents\footnote{We hope to remove this assumption in a later work.} i.e. points $\gamma(s_0),\gamma(s_1)\in\Gamma$ for which $t(s_0)=-t(s_1)$. By compactness, this means that there exists some constant $\delta_0>0$ such that 
\begin{equation}\label{noantipodes}
|t(s)+t(s')|\geq\delta_0,\;\;\forall s,s'\in [0,\ell].
\end{equation}

Certain subsets of $\Gamma$ will be of special interest to us. A cap $\C\subset \Gamma$ is a set of the form
$$\C=\C(s,r)=\{\gamma(s')\in\Gamma: |s-s'|< r\}$$
for some $s\in[0,\ell]$ and $r>0$. We will write $|\C|:=\sigma(\C)$.

The space $L^2(\sigma)$ consists of all functions $f:{\Gamma}\rightarrow\Co$ for which the quantity
$$\|f\|_{L^2(\sigma)}^2:=\int_{{\Gamma}} |f(z)|^2 d\sigma(z)$$
is finite. Given $f\in L^2(\sigma)$, the Fourier transform of the measure $f\sigma$ is defined as
$$\widehat{f\sigma}(x,t):=\int_\Gamma f(z)e^{-i(x,t)\cdot z}d\sigma(z).$$

\noindent The Tomas-Stein inequality \cite{T}, whose proof we recall in the next section, states that there exists a finite constant ${\mathbf C}[\Gamma]<\infty$ for which
\begin{equation}\label{TS}
\|\widehat{f\sigma}\|_{L^6(\R^2)}\leq {\mathbf C}[\Gamma] \|f\|_{L^2(\sigma)}
\end{equation}
for every $f\in L^2(\sigma)$; by ${\mathbf C}[\Gamma]$ we mean the optimal constant defined by 
$${\mathbf C}[\Gamma]:=\sup_{0\neq f\in L^2(\sigma)} \|\widehat{f\sigma}\|_6 \|f\|_{L^2(\sigma)}^{-1}.$$

\begin{definition}
 An {\em extremizing sequence} for the inequality \eqref{TS} is a sequence $\{f_n\}$ of functions in $L^2(\sigma)$ satisfying $\|f_n\|_{L^2(\sigma)}\leq 1$ such that $\|\widehat{f_n\sigma}\|_{6}\rightarrow {\mathbf C}[\Gamma]$ as $n\rightarrow\infty$.
An {\em extremizer} for the inequality \eqref{TS} is  a nonzero function $f\in L^2(\sigma)$ which satisfies $\|\widehat{f\sigma}\|_6= {\mathbf C}[\Gamma]\|f\|_{L^2(\sigma)}$.
\end{definition}

\begin{definition}
A nonzero function $f\in L^2(\sigma)$ is said to be a {\em $\delta$-near extremizer} for the inequality \eqref{TS} if $ \|\widehat{f\sigma}\|_6\geq (1-\delta){\mathbf C}[\Gamma]\|f\|_{L^2(\sigma)}.$
\end{definition}

A natural question is whether extremizers exist. More generally one can ask if extremizing sequences are precompact in $L^2(\sigma)$. Previous work includes the study of extremizers for Strichartz/Fourier restriction inequalities in \cite{K}, \cite{F} and \cite{CS}. Kunze \cite{K} proved the existence of extremizers for the parabola in $\R^2$ by showing that any nonnegative extremizing sequence is precompact. Foschi \cite{F}, whose work will be recalled in greater detail in \S \ref{sec:Comparing optimal constants}, showed that Gaussians are extremizers for this situation and computed the corresponding optimal constant. The best constant and extremizers for the paraboloid in $\R^3$ were also computed in \cite{F}. The existence of extremizers for the restriction on the sphere $S^2$ was proved by Christ and Shao in \cite{CS}, and this is to the best of our knowledge the only result concerning existence of extremizers for the endpoint restriction problem on a compact manifold. 

Other results on (non-)existence of extremizers and/or computation of sharp constants for Fourier restriction operators and Strichartz inequalities can be found in \cite{Ca, DMR, FVV1, FVV2, HZ, Q1, Q2}.

Here is our main result:

\begin{theorem}\label{main}
Let $\Gamma$ be a compact arc of a smooth, convex curve in the plane without points with colinear tangents, equipped with arclength measure $\sigma$. Assume that the curvature $\kappa$ of $\Gamma$ is a strictly positive function. If the second derivative of the curvature with respect to arclength satisfies
\begin{equation}\label{k2} 
\frac{d^2\kappa}{ds^2}(p_0)<\frac{3}{2}\kappa(p_0)^3
\end{equation}
at every $p_0\in\Gamma$ which is a global minimum of the curvature, then any extremizing sequence of nonnegative functions in $L^2(\sigma)$ for the inequality \eqref{TS} is precompact.
\end{theorem}

It is natural to ask about the significance of the geometric condition \eqref{k2}. For instance, if it is not satisfied, does this mean that extremizers fail to exist? While we are at the moment unable to provide a complete answer to this question, we analyze the situation in which condition \eqref{k2} fails in a rather strong sense in the companion paper \cite{OS}, and establish a complementary (negative) result along these lines. 

We conclude this section by briefly outlining the structure of this paper and giving an idea of the proof of Theorem \ref{main}.

In the next section we follow the classical argument of Carleson and Sj\"{o}lin \cite{CSjo} to prove the Tomas-Stein inequality \eqref{TS}. Using the analysis of a certain bilinear form from \cite{MVV, Q1}, we establish the following refinement:
\begin{equation}\label{refinedTS}
\|\widehat{f\sigma}\|_{L^6(\R^2)}\lesssim \|f\|_{L^2(\sigma)}^{1-\beta/2}\sup_{\C\subset\Gamma}\Big(|\C|^{-1/4}\int_\C|f|^{3/2}d\sigma\Big)^{\beta/3},
\end{equation}
where the supremum ranges over all caps $\C\subset\Gamma$ and $\beta>0$ is a small universal constant. 

We use estimate \eqref{refinedTS} in \S \ref{sec:The decomposition algorithm} to describe an iterative procedure which takes a nonnegative function $f\in L^2(\sigma)$ as input and produces a sequence of functions $\{f_n\}$ associated with disjoint caps $\{\C_n\}\subset\Gamma$ for which $f=\sum_n f_n$ in the $L^2$-sense. This decomposition enjoys certain geometric properties which are described in \S \ref{sec:Geometric properties of the decomposition}. After introducing a suitable metric on the set of all caps, we establish the fact that distant caps interact weakly. Together with the decomposition algorithm, this implies an inequality of geometric nature which is a key step towards gaining control of  extremizing sequences.

In \S \ref{sec:Upper bounds for extremizing sequences} we prove that any near extremizer satisfies appropriately scaled upper bounds with respect to some cap, and we use this to obtain a result of {\em concentration compactness} \cite{Li} flavor in \S \ref{sec:A concentration compactness result}. This result basically states that a nonnegative extremizing sequence behaves in one of two possible ways (up to extraction of a subsequence and up to a small $L^2$ error): it is either uniformly integrable, or it concentrates at a point. Precompactness can be derived in the former case, since the main obstruction pointed out in \cite{FVV1} is easy to rule out: in fact, $L^2$ weak limits of nonnegative, uniformly integrable sequences of functions are {\em nonzero}. 

The proof is therefore finished once we show that concentration cannot occur. Aiming at a contradiction, we explore some of the properties that an extremizing sequence which concentrates at a point would have to enjoy. In \S \ref{sec:Exploring concentration}, we compute a certain limiting operator norm exactly, and in particular show that an extremizing sequence which concentrates must do so at a point of minimal curvature.
A second ingredient consists in comparing the constant ${\mathbf C}[\Gamma]$ from inequality \eqref{TS} with the optimal constant for the adjoint Fourier restriction inequality on an appropriately dilated parabola equipped with projection measure, as studied in \cite{F}. We accomplish this in \S \ref{sec:Comparing optimal constants}, postponing some of the more technical estimates to Appendix 1. We derive the desired contradiction in \S \ref{sec:The end of the proof}, and that concludes the proof of Theorem \ref{main}.
\vspace{.5cm}

{\bf Notation.}
If $x,y$ are real numbers, we will write $x=O(y)$ or $x\lesssim y$ if there exists a finite constant $C$ such that $|x|\leq C|y|$, and $x\asymp y$ if $C^{-1}|y|\leq |x|\leq C|y|$ for some finite constant $C\neq 0$. If we want to make explicit the dependence of the constant $C$  on some parameter $\alpha$, we will  write $x=O_\alpha(y)$ or $x\lesssim_\alpha y$. As is customary the constant $C$ is allowed to change from line to line.
If $\lambda\in\R$ and $A\subseteq\R^d$, we denote its $\lambda$-dilation by $\lambda\cdot A:=\{\lambda x: x\in A\}$. The Minkowski sum of $A$ with itself will be denoted by $A+A=\{x+x': x\in A \textrm{ and } x'\in A\}$.
Sharp constants will always appear in bold face.
By $\Re{z}$ and $\Im{z}$ we will denote, respectively, the real and imaginary parts of the complex number $z\in\Co$.
\vspace{.5cm}

{\bf Acknowledgments.} The author would like to thank his dissertation advisor, Michael Christ, for suggesting the problem, and for many helpful discussions and comments. He would also like to thank Ren\'{e} Quilodr\'{a}n for reading an earlier draft of the paper and for several suggestions on how to improve the exposition.

\section{The cap estimate}

 Let $f,g\in L^2(\sigma)$.
We seek to estimate the $L^3$ norm of the product
\begin{equation}\label{fsigmagsigma}
\widehat{f\sigma}\cdot\widehat{g\sigma}(x,t)=\int_0^\ell\int_0^\ell f(\gamma(s))g(\gamma(s')) e^{-i(x,t)\cdot(\gamma(s)+\gamma(s'))}dsds'.
\end{equation}
For that purpose, it will be enough to estimate the $L^{3/2}$ norm of the convolution of measures $f\sigma\ast g\sigma$, which is defined by duality as
$$\langle f\sigma\ast g\sigma, \varphi\rangle=\int_0^\ell\int_0^\ell f(\gamma(s)) g(\gamma(s')) \varphi(\gamma(s)+\gamma(s')) ds ds'$$
for any test function $\varphi\in C^\infty_0(\Gamma,\sigma)$.

To analyze the integral \eqref{fsigmagsigma}, we make the following change of variables:
\begin{equation}\label{changeofvariables}
(s,s')\mapsto(u,v)=(x(s)+x(s'),y(s)+y(s')),
\end{equation}
Splitting
$$f(\gamma(s))g(\gamma(s'))=f(\gamma(s))g(\gamma(s')) (\chi_{\{s>s'\}}+\chi_{\{s<s'\}})\;\textrm{ for a.e. }(s,s')$$
and using the triangle inequality, we lose no generality in assuming that $s>s'$ in the support of $f(\gamma(s))g(\gamma(s'))$. As a consequence, the transformation \eqref{changeofvariables} is injective in the support of $f(\gamma(s)) g(\gamma(s'))$.
It follows that
\begin{equation*}
\widehat{f\sigma}\cdot\widehat{g\sigma}(x,t)=\iint_{\Gamma+\Gamma} f(\gamma(s(u,v)))g(\gamma(s'(u,v))) e^{-i(x,t)\cdot(u,v)} J^{-1} dudv,
\end{equation*}
where by $J=J(s(u,v),s'(u,v))$ we denote the Jacobian of the transformation \eqref{changeofvariables} on the region $\{s>s'\}$:

\begin{equation*}
J(s,s')=\Big|\frac {\partial(u,v)}{\partial(s,s')}\Big|=|x'(s)y'(s')-x'(s')y'(s)|
=|\sin(\theta(s)-\theta(s'))|.
\end{equation*}
Note that, for $(u,v)\in\R^2$,

\begin{displaymath}
f\sigma\ast g\sigma(u,v) = \left\{ \begin{array}{ll}
f(\gamma(s(u,v)))g(\gamma(s'(u,v))) J^{-1}& \textrm{if $(u,v)\in\Gamma+\Gamma$}\\
\qquad\qquad\quad\; 0 & \textrm{otherwise.}
\end{array} \right.
\end{displaymath}

\noindent The Hausdorff-Young inequality implies that
\begin{equation}\label{firstCS}
\|\widehat{f\sigma}\cdot\widehat{g\sigma}\|_3\leq\|f\sigma\ast g\sigma\|_{3/2}\lesssim \Big(\int_0^\ell \int_0^\ell |f(\gamma(s))|^{3/2}|g(\gamma(s'))|^{3/2} |\sin(\theta(s)-\theta(s'))|^{-1/2} ds ds'\Big)^{2/3}.
\end{equation}

\noindent Since $\Gamma$ has no points with colinear tangents (i.e. condition \eqref{noantipodes} holds),
$$|\sin(\theta(s)-\theta(s'))|\geq \min\Big\{\frac{2}{\pi}|\theta(s)-\theta(s')|,\delta_0(1+O(\delta_0^2))\Big\}$$
for every $s,s'\in [0,\ell]$. On the other hand, since $\lambda=\min_{\Gamma}\kappa$,
$$|\theta(s)-\theta(s')|=\Big|\int_{s'}^s\kappa(t)dt\Big|\geq\lambda|s-s'|.$$
It follows that
\begin{align*}
\|f\sigma\ast g\sigma\|_{3/2}^{3/2}&\lesssim \int_0^\ell \int_0^\ell |f(\gamma(s))|^{3/2}|g(\gamma(s'))|^{3/2} |\sin(\theta(s)-\theta(s'))|^{-1/2} ds ds'\\
&\lesssim_{\lambda,\delta_0} \int_0^\ell \int_0^\ell |f(\gamma(s))|^{3/2}|g(\gamma(s'))|^{3/2} |s-s'|^{-1/2} ds ds'.
\end{align*}
Note that the implicit constant blows up as $\lambda\downarrow 0^+$, and this is why we assume that $\Gamma$ has everywhere positive curvature. 

For $0<\alpha<1$, consider the bilinear form:

$$B_\alpha(F,G):=\iint_{\R^2} F(x) G(x') |x-x'|^{-\alpha} dxdx'.$$
The case $\alpha=1/2$ is related to the preceding discussion. In fact, setting $F:=|f\circ\gamma|^{3/2}$ and $G:=|g\circ\gamma|^{3/2}$, we already know that
\begin{equation}\label{CSjo}
\|{f\sigma}\ast{g\sigma}\|_{3/2}^{3/2}\lesssim B_{1/2}(F,G).
\end{equation}

\noindent For $0<\alpha<1$ and $p=2/(2-\alpha)$, the Hardy-Littlewood-Sobolev inequality  implies that $|B_\alpha(F,F)|\lesssim_p \|F\|_{L^p(\R^d)}^2$. 
In particular, estimate \eqref{CSjo} combines with the $L^{4/3}$ bound for $B_{1/2}$ to yield the Tomas-Stein inequality \eqref{TS}: 
$$\|\widehat{f\sigma}\|_6=\|(\widehat{f\sigma})^2\|_3^{1/2}\lesssim\|f\sigma\ast f\sigma\|_{3/2}^{1/2}\lesssim B_{1/2}(F,F)^{1/3}\lesssim \|F\|_{4/3}^{2/3}=\|f\|_{L^2(\sigma)}.$$

Following previous work from \cite{oneil} and \cite{MVV}, Quilodr\'an proved in \cite[Proposition 4.5]{Q1} that, for the same range of $\alpha$ and value of $p$, there exists a  constant $\beta>0$ such that
\begin{equation}\label{Bbounds}
|B_\alpha(F,F)|\lesssim \|F\|_p^{2-\beta}\sup_I\Big(|I|^{-1+1/p}\int_I |F|\Big)^\beta
\end{equation}
for every $F\in L^p(\R)$. Here the supremum ranges over all compact intervals $I$ of $\R$.
If instead of the $L^{4/3}$ bound for $B_{1/2}$ we use the more refined estimate \eqref{Bbounds}, then reasoning in a similar way as before leads to the following improved estimate:

\begin{proposition}\label{capestimate} {\em (Cap estimate)} There exists $C<\infty$ 
and $\beta>0$ such that for every $f\in L^2(\sigma)$, the following estimate holds:
\begin{equation}\label{capestimateinequality}
\|\widehat{f\sigma}\|_{L^6(\R^2)}\leq C \|f\|_{L^2(\sigma)}^{1-\beta/2}\sup_{\C\subset\Gamma}\Big(|\C|^{-1/4}\int_\C|f|^{3/2}d\sigma\Big)^{\beta/3}.
\end{equation}
\end{proposition}
\begin{proof}
Set, as before, $F(s):=|f(\gamma(s))|^{3/2}$. Then \eqref{CSjo} and \eqref{Bbounds} imply:
\begin{align*}
\|\widehat{f\sigma}\|_6&\lesssim\|f\sigma\ast f\sigma\|_{3/2}^{1/2}\lesssim B_{1/2}(F,F)^{1/3}
\lesssim\|F\|_{4/3}^{2/3-\beta/3}\sup_I\Big(|I|^{-1/4}\int_I |F(s)|ds\Big)^{\beta/3}\\
=&\|f\|_{L^2(\sigma)}^{1-\beta/2}\sup_{\C}\Big(|\C|^{-1/4}\int_\C|f|^{3/2}d\sigma\Big)^{\beta/3},
\end{align*}
as desired.
\end{proof}

\section{The decomposition algorithm}\label{sec:The decomposition algorithm}

The cap estimate \eqref{capestimateinequality} is the only ingredient we need to prove the analog of  \cite[{Lemma 2.6}]{CS}, which establishes a weak connection between functions satisfying modest lower bounds $\|\widehat{f \sigma}\|_6\gtrsim\delta\|f\|_2$ and characteristic functions of caps:

\begin{lemma}\label{decomposef}
For any $\delta>0$ there exist $C_\delta<\infty$ and $\eta_\delta>0$ with the following property: if $f\in L^2(\sigma)$ satisfies $\|\widehat{f\sigma}\|_6\geq\delta {\mathbf C}[\Gamma]\|f\|_{L^2(\sigma)}$, then there exists a decomposition $f=g+h$ and a cap $\C\subset\Gamma$ satisfying

\begin{align}
&0\leq |g|, |h|\leq |f|,\label{c1}\\
&g,h\textrm{ have disjoint supports},\label{c2}\\
&|g(\gamma(s))|\leq C_\delta \|f\|_{L^2(\sigma)} |\C|^{-1/2}\chi_\C(\gamma(s)), \textrm{ for all } s\in [0,\ell],\label{c3}\\
&\|g\|_{L^2(\sigma)}\geq\eta_\delta \|f\|_{L^2(\sigma)}\label{c4}.
\end{align}
\end{lemma}

\begin{proof}
The proof is analogous to the one of \cite[Lemma 2.6]{CS} but we reproduce it here for the convenience of the reader. We can, without loss of generality, normalize so that $\|f\|_{L^2(\sigma)}=1$. By Proposition \ref{capestimate} there exists a cap $\C$ such that
$$\int_\C |f|^{3/2}d\sigma\geq \frac{1}{2}c(\delta)|\C|^{1/4}.$$

\noindent Here $c(\delta)=c_0\cdot\delta^{3/\beta}$ for some absolute constant $c_0>0$ whose exact value is not important for the analysis.
Let $R\geq 1$, and define $E:=\{\gamma(s)\in\C: |f(\gamma(s))|\leq R\}$. Set $g=f\chi_E$ and $h=f-f\chi_E$. Then $g$ and $h$ have disjoint supports, $g+h=f$, $g$ is supported on $\C$, and $\|g\|_\infty\leq R$. Since $|h(\gamma(s))|\geq R$ for almost every $\gamma(s)\in \C$ for which $h(\gamma(s))\neq 0$, we have 
$$\int_C |h|^{3/2}d\sigma\leq R^{-1/2}\int_\C |h|^2 d\sigma\leq R^{-1/2}\|f\|_2^2=R^{-1/2}.$$
Define $R$ by $R^{-1/2}=\frac{1}{4}c(\delta)|\C|^{1/4}.$ Then
$$\int_\C |g|^{3/2}d\sigma=\int_\C |f|^{3/2}d\sigma-\int_\C |h|^{3/2}d\sigma\geq \frac{1}{4}c(\delta)|\C|^{1/4}.$$
By H\"{o}lder's inequality, since $g$ is supported on $\C$,
\begin{equation*}
\|g\|_2\geq|\C|^{-1/6}\Big(\int_\C |g|^{3/2}d\sigma\Big)^{2/3}\geq c'(\delta)=c'(\delta)\|f\|_2>0.\qedhere
\end{equation*}
\end{proof}

\noindent Conditions \ref{c3} and \ref{c4} easily imply a lower bound on the $L^1$ norm of $g$:

\begin{lemma}\label{L1lowerbound}
Let $g\in L^2(\sigma)$ satisfy $|g(x)|\leq a |\C|^{-1/2}\chi_\C(x)$ and $\|g\|_2\geq b$ for some $a,b>0$ and $\C\subset\Gamma$. Then there exists a constant $C=C(a,b)>0$ such that
$$\|g\|_{L^1(\sigma)}\geq C|\C|^{1/2}.$$ 
\end{lemma}
\begin{proof}
Estimate:
\begin{equation*}
\|g\|_{L^1(\sigma)}=\int_\C |g|d\sigma\geq a^{-1}|\C|^{1/2}\|g\|_{L^2(\sigma)}^2\geq a^{-1}b^2 |\C|^{1/2}.\qedhere
\end{equation*}
\end{proof}
\vspace{.5cm}

In what follows we will restrict our attention to {\em nonnegative} functions.\footnote{For much of the analysis this makes no difference, but nonnegativity will play a crucial role in \S \ref{sec:A concentration compactness result} when we establish precompactness of uniformly integrable extremizing sequences; see the proof of Lemma \ref{uiimpliesprecompact}.} Indeed, by Plancherel's theorem, inequality \eqref{TS} is equivalent to
\begin{equation}\label{convTS} 
\|f\sigma\ast f\sigma\ast f\sigma\|_{L^2(\R^2)}\leq \frac{{\mathbf C}[\Gamma]^3}{2\pi} \|f\|_{L^2(\sigma)}^3.
\end{equation}
The pointwise inequality $|f\sigma\ast f\sigma\ast f\sigma|\leq |f|\sigma\ast |f|\sigma\ast |f|\sigma$ then implies that, if $f$ is an extremizer for inequality \eqref{TS},  so if $|f|$; similarly, if $\{f_n\}$ is an extremizing sequence,  so is $\{|f_n|\}$.

A decomposition algorithm analogous to the one from \cite[Step 6A]{CS} may be applied to any given  nonnegative $f\in L^2(\sigma)$. We describe it precisely:

\begin{proof}[Decomposition algorithm]
Initialize by setting $G_0=f$ and $\epsilon_0=1/2.$

Step $n$: The inputs for step $n$ are a nonnegative function $G_n\in L^2(\sigma)$ and a positive number $\epsilon_n$. Its outputs are functions $f_n, G_{n+1}$ and nonnegative numbers $\epsilon_n^\star,\epsilon_{n+1}$. 

If $\|G_{n}\sigma\ast G_{n}\sigma\ast G_{n}\sigma\|_2=0$, then $G_n=0$ almost everywhere. The algorithm then terminates, and we define $\epsilon_n^\star=0$, $f_n=0$, and $G_m=f_m=0$, $\epsilon_m=0$ for all $m>n$.

If $0<\|G_{n}\sigma\ast G_{n}\sigma\ast G_{n}\sigma\|_2<\epsilon_n^3 (2\pi)^{-1} {\mathbf C}[\Gamma]^3 \|f\|_2^3$, then replace $\epsilon_n$ by $\epsilon_n/2$; repeat until the first time that $\|G_{n}\sigma\ast G_{n}\sigma\ast G_{n}\sigma\|_2\geq \epsilon_n^3 (2\pi)^{-1} {\mathbf C}[\Gamma]^3 \|f\|_2^3$. Define $\epsilon_n^\star$ to be this value of $\epsilon_n$. Then

\begin{equation}\label{algstep}
(\epsilon_n^\star)^3 \frac{{\mathbf C}[\Gamma]^3}{2\pi}  \|f\|_2^3\leq \|G_{n}\sigma\ast G_{n}\sigma\ast G_{n}\sigma\|_2\leq 8(\epsilon_n^\star)^3 \frac{{\mathbf C}[\Gamma]^3}{2\pi} \|f\|_2^3.
\end{equation}

Apply Lemma \ref{decomposef} to obtain a cap $\C_n$ and a decomposition $G_n=f_n+G_{n+1}$
with disjointly supported nonnegative summands satisfying $f_n\leq C_n \|f\|_2 |\C_n|^{-1/2}\chi_{\C_n}$ and $\|f_n\|_2\geq\eta_n \|f\|_2$. Here, $C_n,\eta_n$ are bounded above and below, respectively, by quantities which depend only on $\|G_{n}\sigma\ast G_{n}\sigma\ast G_{n}\sigma\|_2^{1/3}/\|G_n\|_2\gtrsim \epsilon_n^\star$. Define $\epsilon_{n+1}=\epsilon_n^\star$, and move on to step $n+1$.
\end{proof}

\noindent The following exact analogs of  \cite[Lemmas 8.1, 8.3, 8.4]{CS} hold:
\begin{lemma}\label{decalg1}
Let $f\in L^2(\sigma)$ be a nonnegative function with positive norm. If the decomposition algorithm never terminates for $f$, then $\epsilon_n^\star\rightarrow 0$ as $n\rightarrow\infty$, and $\sum_{n=0}^N f_n\rightarrow f$ in $L^2(\sigma)$ as $N\rightarrow\infty$.  
\end{lemma}

 The proof of Lemma \ref{decalg1} is identical to the corresponding one in \cite{CS} and therefore is omitted.

This decomposition is in general very inefficient. However, if $f$ nearly extremizes inequality \eqref{TS}, then more useful properties hold. Before we turn into these, let us recall a useful fact about near extremizers which already appeared in \cite[Lemma 9.2]{ThQ}:

\begin{lemma}\label{nearextremizers}
Let $f=g+h\in L^2(\sigma)$. Suppose that $g\perp h$, $g\neq 0$, and that $f$ is $\delta$-near extremizer for some $\delta\in (0,\frac{1}{4}]$. Then
\begin{equation}\label{firstprinciple}
\frac{\|h\|_2}{\|f\|_2}\leq C\max\Big(\frac{\|\widehat{h\sigma}\|_6}{\|h\|_2},\delta^{1/2}\Big).
\end{equation}
Here $C<\infty$ is a constant independent of $g$ and $h$.
\end{lemma}

\noindent The proof is almost identical to that of \cite[Lemma 7.1]{CS} but we reproduce it here for the convenience of the reader.

\begin{proof}
The inequality is invariant under multiplication of $f$ by a positive constant, so we may assume without loss of generality that $\|g\|_2=1$. We may assume that $\|h\|_2>0$, since otherwise the conclusion is trivial. Define $y=\|h\|_2$ and 
$$\eta=\frac{\|\widehat{h\sigma}\|_6}{{\mathbf C}[\Gamma] \|h\|_2}.$$ 
If $\eta>\frac{1}{2}$ then \eqref{firstprinciple} holds trivially with $C=2/{\mathbf C}[\Gamma]$, for the left-hand side cannot exceed 1 since $f=g+h$ with $g\perp h$.

We also have that
$$(1-\delta){\mathbf C}[\Gamma] \|f\|_2\leq \|\widehat{f\sigma}\|_6\leq \|\widehat{g\sigma}\|_6+\|\widehat{h\sigma}\|_6\leq{\mathbf C}[\Gamma] (1+\eta y).$$

\noindent Since $g\perp h$, $\|f\|_2^2=1+y^2$ and therefore
$$(1-\delta)(1+y^2)^{1/2}\leq 1+\eta y.$$

\noindent Squaring gives
$$(1-2\delta)(1+y^2)\leq 1+2\eta y+\eta^2 y^2.$$

\noindent Since $\delta\in (0,\frac{1}{4}]$ and $\eta\leq\frac{1}{2}$,
$$\frac{1}{2}y^2\leq 2\delta+2\eta y+\eta^2 y^2\leq 2\delta+2\eta y+\frac{1}{4}y^2$$
whence either $y^2\leq 16\delta$ or $y\leq 16\eta$.

Substituting the definitions of $y, \eta$, and majorizing $\|h\|_2/\|f\|_2$ by $\|h\|_2/\|g\|_2$, yields the stated conclusion.
\end{proof}

 Regardless of whether the decomposition algorithm terminates for $f$, the norms of $f_n, G_n$ enjoy upper bounds independent of $f$, for all but very large $n$:

\begin{lemma}\label{decalg3}
There exist a sequence of positive constants $\gamma_n{\rightarrow} 0$ and a function $N:(0,\frac{1}{2}]\rightarrow\N$ satisfying $N(\delta)\rightarrow\infty$ as $\delta\rightarrow 0$ with the following property: for any nonnegative $\delta$-near extremizer $f\in L^2(\sigma)$, the quantities $\epsilon_n^\star$ and the functions $f_n, G_n$ obtained when the decomposition algorithm is applied to $f$ satisfy
\begin{align}
&\epsilon_n^\star\leq\gamma_n\textrm{ for all }n\leq N(\delta),\label{ind1}\\
&\|G_n\|_2\leq\gamma_n \|f\|_2\textrm{ for all }n\leq N(\delta),\textrm{and}\label{ind2}\\
&\|f_n\|_2\leq\gamma_n \|f\|_2\textrm{ for all }n\leq N(\delta).\label{ind3}
\end{align}
\end{lemma}

\begin{proof} By \eqref{convTS} and \eqref{algstep},
$$\frac{{\mathbf C}[\Gamma]^3}{2\pi} \|G_n\|_2^3\geq \|G_n\sigma\ast G_n\sigma\ast G_n\sigma\|_2\geq (\epsilon_n^\star)^3\frac{{\mathbf C}[\Gamma]^3}{2\pi}\|f\|_2^3=\Big(\frac{(\epsilon_n^\star)^3\|f\|_2^3}{\|G_n\|_2^3}\Big)\cdot\frac{{\mathbf C}[\Gamma]^3}{2\pi}\|G_n\|_2^3,$$
so $\epsilon_n^\star\leq \|G_n\|_2/\|f\|_2$. Thus the second conclusion implies the first. Since $\|f_n\|_2\leq\|G_n\|_2$, it also implies the third. 

We recall two facts. Firstly, Lemma \ref{nearextremizers}  applied to $h=G_n$ and $g=f_0+\ldots+f_{n-1}$ asserts that there are constants $c_0,C_1\in\R^+$ such that whenever $f\in L^2$ is a $\delta$-near extremizer, either $\|\widehat{G_n\sigma}\|_6\geq c_0 \|G_n\|_2^2 \|f\|_2^{-1}$, or $\|G_n\|_2\leq C_1\delta^{1/2}\|f\|_2$.  Secondly, according to Lemma \ref{decomposef}, there exists a nondecreasing function $\rho:(0,\infty)\rightarrow (0,\infty)$ satisfying $\rho(t)\rightarrow 0$ as $t\rightarrow 0$ such that for every nonzero $f\in L^2$ and any $n$, if $\|\widehat{G_n\sigma}\|_6\geq t\|G_n\|_2$, then $\|f_n\|_2\geq \rho(t)\|G_n\|_2$.

Choose a sequence $\{\gamma_n\}$ of positive numbers which tends monotonically to zero, but does so sufficiently slowly to satisfy
$$(n+1)\gamma_n\rho(c_0\gamma_n)>1\textrm{ for all }n.$$
Define $N(\delta)$ to be the largest integer satisfying 
$$\gamma_{N(\delta)}\geq C_1\delta^{1/2}.$$
Note that $N(\delta)\rightarrow\infty$ as $\delta\rightarrow 0$ because $\gamma_n> 0$ for all $n$.

Let $f,\delta$ be given. Suppose that $n\leq N(\delta)$. Aiming at a contradiction, suppose that $\|G_n\|_2>\gamma_n \|f\|_2$. Then by definition of $N(\delta)$, $\|G_n\|_2>C_1\delta^{1/2}\|f\|_2.$
By the above dichotomy,
$$\|\widehat{G_n\sigma}\|_6\geq c_0 \|G_n\|_2^2 \|f\|_2^{-1}\geq c_0\gamma_n \|G_n\|_2.$$

By the second fact reviewed above, 
$$\|f_n\|_2\geq \rho(c_0\gamma_n)\|G_n\|_2\geq \gamma_n\rho(c_0\gamma_n)\|f\|_2.$$
Since $\|G_m\|_2\geq \|G_n\|_2$ for every $m\leq n$, the same lower bound follows for $\|f_m\|_2$ for every $m\leq n$. 
Since the functions $f_m$ are pairwise orthogonal, $\sum_{m\leq n} \|f_m\|_2^2\leq \|f\|_2^2$, and consequently $(n+1)\gamma_n\rho(c_0\gamma_n)\leq 1$, a contradiction.
\end{proof}

\noindent The following lemma is a direct consequence of the decomposition algorithm coupled with Lemma \ref{decomposef}:

\begin{lemma}\label{decalg4}
For every $\epsilon>0$ there exist $\delta_\epsilon>0$ and $C_\epsilon<\infty$ such that, if $f\in L^2(\sigma)$ is a nonnegative $\delta_\epsilon$-near extremizer, then the functions $f_n, G_n$ associated to $f$ by the decomposition algorithm satisfy, for every $n\in\N$,
\begin{itemize}
\item[(i)] If $\|G_n\|_2\geq\epsilon\|f\|_2$ then there exists a cap $\C_n\subset\Gamma$ such that
$$f_n\leq C_\epsilon \|f\|_2 |\C_n|^{-1/2}\chi_{\C_n}.$$
\item[(ii)] If $\|G_n\|_2\geq\epsilon\|f\|_2$, then $\|f_n\|_2\geq\delta_\epsilon \|f\|_2$.
\end{itemize}
\end{lemma}

\section{A geometric property of the decomposition}\label{sec:Geometric properties of the decomposition}

Consider two caps $\C,\C'\subset\Gamma$, and assume without loss of generality that $|\C'|\leq |\C|$. Let $f,g\in L^2(\sigma)$ be such that supp$(f)\subset\C$ and supp$(g)\subset\C'$. The following estimate is a direct consequence of \eqref{firstCS} and H\"{o}lder's inequality :
\begin{equation}\label{far}
\|f\sigma\ast g\sigma\|_{3/2}\lesssim\Big(\frac{\inf_{s,s'}|\sin(\theta(s)-\theta(s'))|}{|\C|^{1/2}|\C'|^{1/2}}\Big)^{-1/3}\|f\|_{L^2(\sigma)}\|g\|_{L^2(\sigma)}.
\end{equation}
We can rewrite this estimate in the following way: letting $\ell(\C,\C'):=\inf_{s\in\C,s'\in\C'}|s-s'|$, then \eqref{CSjo} and H\"{o}lder's inequality imply
\begin{equation}\label{far2}
\|f\sigma\ast g\sigma\|_{3/2}\lesssim\Big(\frac{\ell(\C,\C')}{|\C|^{1/2}|\C'|^{1/2}}\Big)^{-1/3}\|f\|_{L^2(\sigma)}\|g\|_{L^2(\sigma)}.
\end{equation}

For characteristic functions of caps we have the following additional estimate:
\begin{lemma}\label{estimatingcaps} 
Let $\C,\C'\subset\Gamma$ be caps. Then:
\begin{equation}\label{different}
\|\chi_\C\sigma\ast\chi_{\C'}\sigma\|_{L^{3/2}(\R^2)}\lesssim \Big(\frac{|\C'|}{|\C|}\Big)^{1/12} |\C|^{1/2}|\C'|^{1/2}.
\end{equation}
\end{lemma}

\begin{proof}
Without loss of generality we might assume that $10|\C'|\leq |\C|$, otherwise estimate \eqref{different} is just a consequence of the fundamental inequality 

\begin{equation}\label{HLS}
\|f\sigma\ast g\sigma\|_{3/2}\lesssim \|f\|_{L^2(\sigma)}\|g\|_{L^2(\sigma)}.
\end{equation}
Let $\C^*$ be a cap neighborhood of $\C'$ with the same center and of size $|\C^*|=|\C|^{3/4}|\C'|^{1/4}$. Split $\chi_\C=\chi_{\C\cap \C^*}+\chi_{\C\setminus\C^*}$. Then

$$\|\chi_\C\sigma\ast\chi_{\C'}\sigma\|_{3/2}\leq\|\chi_{\C\cap\C^*}\sigma\ast\chi_{\C'}\sigma\|_{3/2}+\|\chi_{\C\setminus \C^*}\sigma\ast\chi_{\C'}\sigma\|_{3/2}.$$

\noindent 
The first summand can be easily estimated using \eqref{HLS}.
While $\C\setminus\C^*$ is not necessarily a cap, it is the union of at most two caps, say, $\C_1$ and $\C_2$. 
We can use estimate \eqref{far2} to control the contribution of each of these caps. 
Noting that 
$$\min\{\ell(\C_1,\C'),\ell(\C_2,\C')\}\geq\ell(\C\setminus\C^*,\C')\gtrsim |\C^*|,$$ 
we have that
\begin{align*}
\|\chi_\C\sigma\ast\chi_{\C'}\sigma\|_{3/2}
\lesssim& \Big(|\C\cap\C^*|^{1/2}+\Big(\frac{\ell(\C_1,\C')}{|\C_1|^{1/2}|\C'|^{1/2}}\Big)^{-1/3}|\C_1|^{1/2}+\Big(\frac{\ell(\C_2,\C')}{|\C_2|^{1/2}|\C'|^{1/2}}\Big)^{-1/3}|\C_2|^{1/2}\Big)|\C'|^{1/2}\\
\lesssim& \Big(|\C^*|^{1/2}+\Big(\frac{|\C^*|}{|\C|^{1/2}|\C'|^{1/2}}\Big)^{-1/3}|\C|^{1/2}\Big)|\C'|^{1/2}\\
\lesssim& \Big(\frac{|\C'|}{|\C|}\Big)^{1/12} |\C|^{1/2}|\C'|^{1/2},
\end{align*}  
as desired.
\end{proof}

The set of all caps can be made into a metric space.
We define the distance $d$ from $\C=\C(s,r)$ to $\C'=\C(s',r')$ to be the hyperbolic distance from $(s,r)$ to $(s',r')$ in the upper half plane model. More explicitly, we have that
\begin{equation}\label{hypdist}
d(\C,\C'):=\textrm{arc}\cosh\Big(1+\frac{(s-s')^2+(r-r')^2}{2rr'}\Big).
\end{equation}
If $s=s'$, then the distance depends only on the ratio of the two radii.
When $r=r'$, the distance is $\asymp r^{-1}|s-s'|$ and so this distance has the natural scaling. 

We can use estimates \eqref{far2} and \eqref{different} to prove that  the quantity $\|\chi_\C\sigma\ast\chi_{\C'}\sigma\|_{3/2}$ is much smaller than the trivial bound $|\C|^{1/2}|\C'|^{1/2}$ unless $\C,\C'$ have comparable radii and nearby centers:

\begin{lemma}\label{widc}
For any $\epsilon>0$ there exists $\rho<\infty$ such that
$$\|\chi_\C\sigma\ast\chi_{\C'}\sigma\|_{L^{3/2}(\R^2)}<\epsilon |\C|^{1/2}|\C'|^{1/2}$$
whenever
$$d(\C,\C')>\rho.$$
\end{lemma}

\begin{proof}
Let $\C=\C(s,r)$ and $\C'=\C'(s',r')$. As before, assume $r'\leq r$. We consider three cases.

Start by assuming that $\C$ and $\C'$ have comparable radii: say, $\frac{1}{10}r\leq r'\leq r.$ Then $\C$ and $\C'$ are not far apart unless the corresponding centers are far apart. We may therefore assume that $|s-s'|\geq 10r$, which in turn implies $\ell(\C,\C')\gtrsim |s-s'|$. Using estimate \eqref{far2}, we conclude that

$$\|\chi_\C\sigma\ast\chi_{\C'}\sigma\|_{3/2}\lesssim\Big(\frac{|s-s'|}{r}\Big)^{-1/3}|\C|^{1/2}|\C'|^{1/2}\lesssim (\cosh\rho)^{-1/6}|\C|^{1/2}|\C'|^{1/2}$$
provided $d(\C,\C')>\rho$.

Assume now that $10r'<r$. 
If $|s-s'|<10r$, then we can use Lemma \ref{estimatingcaps} to conclude

$$\|\chi_\C\sigma\ast\chi_{\C'}\sigma\|_{3/2}\lesssim\Big(\frac{r'}{r}\Big)^{1/12} |\C|^{1/2}|\C'|^{1/2}.$$
This quantity is $\lesssim(\cosh\rho)^{-1/12}|\C|^{1/2}|\C'|^{1/2}$ provided $d(\C,\C')>\rho$, and this concludes the analysis in this case.
If on the other hand $|s-s'|\geq10r$, then $d(\C,\C')>\rho$ implies
$$\frac{(s-s')^2}{rr'}\gtrsim \cosh\rho\;\textrm{ or }\;\frac{r}{r'}\gtrsim\cosh\rho.$$
Using, as before, estimate \eqref{far2} in the former case and Lemma \ref{estimatingcaps} in the latter, we arrive at the desired conclusion.
\end{proof}

\noindent For applications later on, we will need a trilinear version of this lemma which follows immediately from the previous result:

\begin{corollary}\label{triwidc}
For any $\epsilon>0$ there exists $\rho<\infty$ such that
$$\|\chi_\C\sigma\ast\chi_{\C'}\sigma\ast\chi_{\C''}\sigma\|_{L^2(\R^2)}<\epsilon |\C|^{1/2}|\C'|^{1/2}|\C''|^{1/2}$$
whenever
$$\max\{d(\C,\C'), d(\C',\C''), d(\C'',\C)\}>\rho.$$
\end{corollary}

\begin{proof}
 Using Cauchy-Schwarz,
 \begin{equation}\label{tripleCS}
\|\chi_\C\sigma\ast\chi_{\C'}\sigma\ast\chi_{\C''}\sigma\|_{2}\leq\|\chi_\C\sigma\ast\chi_{\C}\sigma\ast\chi_{\C'}\sigma\|_{2}^{1/2}\|\chi_{\C'}\sigma\ast\chi_{\C''}\sigma\ast\chi_{\C''}\sigma\|_{2}^{1/2}.
\end{equation}
Without loss of generality we may assume that the caps $\C$ and $\C'$ are far apart: in view of Lemma \ref{widc}, we can choose $\rho<\infty$ so that $d(\C,\C')>\rho$ implies $\|{\chi_\C\sigma}\ast{\chi_{\C'}\sigma}\|_{3/2}<\epsilon^2 |\C|^{1/2}|\C'|^{1/2}$. Bound the first factor appearing in \eqref{tripleCS} as follows:
\begin{align*}
\|\chi_\C\sigma\ast\chi_{\C}\sigma\ast\chi_{\C'}\sigma\|_{2}&\simeq\|\widehat{\chi_\C\sigma}\cdot\widehat{\chi_\C\sigma}\cdot\widehat{\chi_{\C'}\sigma}\|_2\leq \|\widehat{\chi_\C\sigma}\|_6\|\widehat{\chi_\C\sigma}\cdot\widehat{\chi_{\C'}\sigma}\|_3\\
&\leq\|\widehat{\chi_\C\sigma}\|_6\|{\chi_\C\sigma}\ast{\chi_{\C'}\sigma}\|_{3/2}<\epsilon^2 |\C||\C'|^{1/2}.
\end{align*}
The proof is then complete in view of  the trivial estimate
$\|\chi_{\C'}\sigma\ast\chi_{\C''}\sigma\ast\chi_{\C''}\sigma\|_{2}\leq |\C'|^{1/2}|\C''|.$
\end{proof}

\noindent Corollary \ref{triwidc} allows us to establish the following additional inequality of geometric nature, which can be proved in an identical way to \cite[Lemma 2.38]{ThQ}. 
\begin{lemma}\label{geoppty}
For any $\epsilon>0$ there exist $\delta>0$ and $\lambda<\infty$ such that for any nonnegative $ f\in L^2(\sigma)$ which is $\delta$-near extremizer, the summands $f_n$ produced by the decomposition algorithm and the associated caps $\C_n$ satisfy 
$$d(\C_j,\C_k)\leq\lambda\textrm{ whenever }\|f_j\|_2\geq\epsilon\|f\|_2\textrm{ and }\|f_k\|_2\geq\epsilon\|f\|_2.$$
\end{lemma}

\noindent We provide one proof which follows the proof of \cite[Lemma 9.2]{CS} more closely:
\begin{proof}
It suffices to prove this for all sufficiently small $\epsilon>0$. Let $f$ be a nonnegative $L^2$ function which satisfies $\|f\|_2=1$ and is $\delta$-near extremizer for a sufficiently small $\delta=\delta(\epsilon)$, and let $\{f_n,G_n\}$ be associated to $f$ via the decomposition algorithm. Set $F=\sum_{n=0}^N f_n$.

Suppose that $\|f_{j_0}\|_2\geq\epsilon$ and $\|f_{k_0}\|_2\geq\epsilon$. Let $N$ be the smallest integer such that $\|G_{N+1}\|_2<\epsilon^3$. Since $\|G_n\|_2$ is a nonincreasing function of $n$, and since $\|f_n\|_2\leq \|G_n\|_2$, necessarily $j_0,k_0\leq N$. Moreover, by Lemma \ref{decalg3}, there exists $M_\epsilon<\infty$ depending {\em only} on $\epsilon$ such that $N\leq M_\epsilon$. By Lemma \ref{decalg4}, if $\delta$ is chosen to be a sufficiently small function of $\epsilon$ then since $\|G_n\|_2\geq\epsilon^3$ for all $n\leq N$, $f_n\leq\theta(\epsilon) |\C_n|^{-1/2}\chi_{\C_n}$ for all such $n$, where $\theta$ is a continuous, strictly positive function on $(0,1]$.

Now let $\lambda<\infty$ be a large quantity to be specified. It suffices to show that if $\delta(\epsilon)$ is sufficiently small, an assumption that $d(\C_j,\C_k)>\lambda$ implies an upper bound, which depends only on $\epsilon$, for $\lambda$.

As proved in \cite[Lemma 9.1]{CS}, there exists a decomposition $F=F_1+F_2=\sum_{n\in S_1} f_n+\sum_{n\in S_2} f_n$ where $[0,N]=S_1\cup S_2$ is a partition of $[0,N]$, $j_0\in S_1$, $k_0\in S_2$, and $d(\C_j,\C_k)\geq\lambda/2N\geq\lambda/2M_\epsilon$ for all $j\in S_1$ and $k\in S_2$. Certainly $\|F_1\|_2\geq \|f_{j_0}\|_2\geq\epsilon$, and similarly $\|F_2\|_2\geq\epsilon$. One of the cross term satisfies

\begin{align*}
\|F_1\sigma\ast F_1\sigma\ast F_2\sigma\|_2&\leq \sum_{i\in S_1}\sum_{j\in S_1}\sum_{k\in S_2} \|f_i\sigma\ast f_j\sigma\ast f_k\sigma\|_2\\
&\leq\theta(\epsilon)^3 \sum_{i\in S_1}\sum_{j\in S_1}\sum_{k\in S_2} |\C_i|^{-1/2}|\C_j|^{-1/2}|\C_k|^{-1/2}\|\chi_{\C_i}\sigma\ast\chi_{\C_j}\sigma\ast\chi_{\C_k}\sigma\|_2\\
&\leq M_\epsilon^3\gamma(\lambda/2M_\epsilon)\theta(\epsilon)^3
\end{align*}
where $\gamma(\lambda)\rightarrow 0$ as $\lambda\rightarrow\infty$ by Corollary \ref{triwidc}.
The other cross term $F_1\sigma\ast F_2\sigma\ast F_2\sigma$ can be estimated in an identical way. It follows that

\begin{align*}
\|F\sigma\ast F\sigma\ast F\sigma\|_2&\leq\|F_1\sigma\ast F_1\sigma\ast F_1\sigma\|_2+\|F_2\sigma\ast F_2\sigma\ast F_2\sigma\|_2+3\|F_1\sigma\ast F_1\sigma\ast F_2\sigma\|_2+\\
&+3\|F_1\sigma\ast F_2\sigma\ast F_2\sigma\|_2\\
&\leq (2\pi)^{-1}{\mathbf C}[\Gamma]^3(\|F_1\|_2^3+\|F_2\|_2^3)+6M_\epsilon^3\gamma(\lambda/2M_\epsilon)\theta(\epsilon)^3.
\end{align*}
Since $F_1$ and $F_2$ have disjoint supports, $\|F_1\|_2^2+\|F_2\|_2^2\leq\|f\|_2^2=1$ and consequently
$$\|F_1\|_2^3+\|F_2\|_2^3\leq\max(\|F_1\|_2,\|F_2\|_2)\cdot (\|F_1\|_2^2+\|F_2\|_2^2)\leq (1-\epsilon^2)^{1/2}\cdot 1\leq (1-\epsilon^2)^{1/2}.$$
Thus
$$\|F\sigma\ast F\sigma\ast F\sigma\|_2\leq (2\pi)^{-1}{\mathbf C}[\Gamma]^3(1-\epsilon^2)^{1/2}+6M_\epsilon^3\gamma(\lambda/2M_\epsilon)\theta(\epsilon)^3.$$
Since
\begin{align*}
(2\pi)^{-1}{\mathbf C}[\Gamma]^3 (1-\delta)^3\leq \|f\sigma\ast f\sigma\ast f\sigma\|_2&\leq \|F\sigma\ast F\sigma\ast F\sigma\|_2+C\|F\|_2^2\|G_{N+1}\|_2\\
&\leq \|F\sigma\ast F\sigma\ast F\sigma\|_2+C\epsilon^3,
\end{align*}
by transitivity we have that

$$(2\pi)^{-1}{\mathbf C}[\Gamma]^3 (1-\delta)^3\leq C\epsilon^3+(2\pi)^{-1}{\mathbf C}[\Gamma]^3(1-\epsilon^2)^{1/2}+6M_\epsilon^3\gamma(\lambda/2M_\epsilon)\theta(\epsilon)^3.$$

\noindent Since $\gamma(t)\rightarrow 0$ as $t\rightarrow \infty$, this implies, for all sufficiently small $\epsilon>0$, an upper bound for $\lambda$ which depends only on $\epsilon$, as was to be proved.
\end{proof}

\section{Upper bounds for extremizing sequences}\label{sec:Upper bounds for extremizing sequences}

The decomposition algorithm and Lemma \ref{geoppty} allow us to prove that any near extremizer satisfies appropriate scaled upper bounds with respect to some cap. First we need a definition:

\begin{definition}
Let $\Theta:[1,\infty)\rightarrow (0,\infty)$ satisfy $\Theta(R)\rightarrow 0$ as $R\rightarrow\infty$.
A function $f\in L^2(\sigma)$ is said to be {\em upper normalized} (with gauge function $\Theta$) with respect to a cap $\C=\C(\gamma(s_0),r_0)\subset\Gamma$ of radius $r_0$ and center $\gamma(s_0)$ if
\begin{align}
&\|f\|_{L^2(\sigma)}\leq C<\infty,\label{un1}\\
&\int_{\{s: |f(\gamma(s))|\geq R r_0^{-1/2}\}}|f(\gamma(s))|^2ds\leq\Theta(R),\;\;\forall R\geq 1,\label{un2}\\
&\int_{\{s: |s-s_0|\geq Rr_0\}} |f(\gamma(s))|^2ds\leq\Theta(R),\;\;\forall R\geq 1\label{un3}.
\end{align} 
\end{definition}

\begin{proposition}\label{upperbounds}
There exists a function $\Theta:[1,\infty)\rightarrow (0,\infty)$ satisfying $\Theta(R)\rightarrow 0$ as $R\rightarrow\infty$ with the following property. For every $\epsilon>0$, there exist a cap $\C\subset\Gamma$ and a threshold  $\delta>0$ such that any nonnegative $f\in L^2(\sigma)$ which is a $\delta$-near extremizer with $\|f\|_2=1$ may be decomposed as $f=F+G$, where:
\begin{align}
&G,F\geq 0\textrm{ have disjoint supports},\label{bds1}\\
&\|G\|_2<\epsilon,\label{bds2}\\
&F\textrm{ is upper normalized with respect to }\C.\label{bds3}
\end{align}
\end{proposition}

\noindent Proposition \ref{upperbounds} is actually equivalent to the following superficially weaker statement:

\begin{lemma}\label{upperboundsR}
There exists a function $\Theta:[1,\infty)\rightarrow (0,\infty)$ satisfying $\Theta(R)\rightarrow 0$ as $R\rightarrow\infty$ with the following property. For every $\epsilon>0$ and $\bar{R}\geq 1$, there exist a cap $\C=\C(s_0,r_0)$ and a threshold $\delta>0$ such that any nonnegative $f\in L^2(\sigma)$ which is a $\delta$-near extremizer with $\|f\|_2=1$ may be decomposed as $f=F+G$, where:
\begin{align}
&G,F\geq 0\textrm{ have disjoint supports},\label{bdsR1}\\
&\|G\|_2<\epsilon,\label{bdsR2}\\
&\int_{\{s: F(\gamma(s))\geq R r_0^{-1/2}\}}F(\gamma(s))^2ds,\int_{\{s: |s-s_0|\geq Rr_0\}} F(\gamma(s))^2ds\leq\Theta(R),\;\; \forall R\in[1,\bar{R}].\label{bdsR3}
\end{align}
\end{lemma}

\noindent Proof that Lemma \ref{upperboundsR} implies Proposition \ref{upperbounds} is the exactly the same as in \cite[p. 26]{CS} and so we do not include it here. 

\begin{proof}[Proof of Lemma \ref{upperboundsR}]
Let $\eta:[1,\infty)\rightarrow (0,\infty)$ be a function to be chosen below, satisfying $\eta(t)\rightarrow 0$ as $t\rightarrow\infty$. This function will not depend on the quantity $\bar{R}$.

Let $\bar{R}\geq 1$, $R\in [1,\bar{R}]$, and $\epsilon>0$ be given. Let $\delta=\delta(\epsilon,\bar{R})>0$ be a small quantity to be chosen below. Let $0\leq f\in L^2(\sigma)$ be $\delta$-near extremizer, with $\|f\|_2=1$.

Let $\{f_n\}$ be the sequence of functions obtained by applying the decomposition algorithm to $f$. Choose $\delta=\delta(\epsilon)>0$ sufficiently small and $M=M(\epsilon)$ sufficiently large to guarantee that $\|G_{M+1}\|_2<\epsilon/2$ and that $f_n, G_n$ satisfy all the conclusions of Lemma \ref{decalg4} and Lemma \ref{decalg3} for $n\leq M$. Set $F=\sum_{n=0}^M f_n$. Then $\|f-F\|_2=\|G_{M+1}\|_2<\epsilon/2$.

Let $N\in\N_0$ be the minimum of $M$, and the smallest number such that $\|f_{N+1}\|_2<\eta.$ $N$ is majorized by a quantity which depends only on $\eta$. Set $\F=\F_N=\sum_{k=0}^N f_k$. It follows from part $(ii)$ of Lemma \ref{decalg4} that

\begin{equation}\label{boundF-F}
\|F-\F\|_2\leq\|G_{N+1}\|_2<\zeta(\eta)\textrm{ where } \zeta(\eta)\rightarrow 0\textrm{ as }\eta\rightarrow 0.
\end{equation}
This function $\zeta$ is independent of $\epsilon$ and $\bar{R}$.

To prove the lemma, we must produce an appropriate cap $\C=\C(s_0,r_0)$, and must establish the existence of $\Theta$. To do the former is simple: to $f_0$ is associated a cap $\C_0$ such that $f_0\leq C|\C_0|^{-1/2}\chi_{\C_0}$. Then $\C=\C_0$ is the required cap. Note that, by Lemma \ref{L1lowerbound}, $\|f_0\|_1\geq c$ for some positive universal constant $c$.

Suppose that functions $R\mapsto \eta(R)$ and $R\mapsto\Theta(R)$ are chosen so that

\begin{align}
&\eta(R)\rightarrow 0\textrm{ as }R\rightarrow\infty\\
&\zeta(\eta(R))^2\leq\Theta(R)\textrm{ for all }R.
\end{align}

\noindent Then, by \eqref{boundF-F}, $F-\F$ already satisfies the desired inequalities in $L^2(\sigma)$, so it suffices to show that $\F(\gamma(s))\equiv 0$ whenever $|s-s_0|\geq Rr_0$, and that $\|\F\|_\infty< Rr_0^{-1/2}$.

Each summand of $\F$ satisfies $f_k\leq C(\eta)|\C_k|^{-1/2}\chi_{\C_k}$ where $C(\eta)<\infty$ depends only on $\eta$, and in particular, $f_k$ is supported in $\C_k$. $\|f_k\|_2\geq\eta$ for all $k\leq N$, by definition of $N$. Therefore by Lemma \ref{geoppty}, there exists a function $\eta\mapsto\lambda(\eta)<\infty$, such that if $\delta$ is sufficiently small as a function of $\eta$ then $d(\C_k,\C_0)\leq\lambda(\eta)$ for every $k\leq N$. This is needed for $\eta=\eta(R)$ for all $R$ in the compact set $[1,\bar{R}]$, so such a $\delta$ may be chosen as a function of $\bar{R}$ alone; conditions already imposed on $\delta$ above make it a function of both $\epsilon$ and $\bar{R}$. 

In the region of all $\gamma(s)\in\Gamma$ such that $|s-s_0|\geq Rr_0$, either $f_k\equiv 0$, or $\C_k$ has radius $\geq \frac{1}{2}Rr_0$, or the center $\gamma(s_k)$ of $\C_k$ is such that $|s_k-s_0|\geq \frac{1}{2}Rr_0$. Choose a function $R\mapsto\eta(R)$ which tends to 0 sufficiently slowly that the latter two cases would contradict the inequality $d(\C_k,\C_0)\leq\lambda$, and therefore cannot arise. Then $\F(\gamma(s))\equiv 0$ when $|s-s_0|\geq Rr_0$.

With the function $\eta$ specified, $\Theta$ can be defined by
$\Theta(R):=\zeta(\eta(R))^2.$
Then
\begin{equation}\label{req1}
\int_{\{s: |s-s_0|\geq Rr_0\}} F(\gamma(s))^2ds\leq\Theta(R),\;\; \forall R\in[1,\bar{R}].
\end{equation}

We claim next that $\|\F\|_\infty<Rr_0^{-1/2}$ if $R$ is sufficiently large as a function of $\eta$. Indeed, because the summands $f_k$ have pairwise disjoint supports, it suffices to control $\max_{k\leq N}\|f_k\|_\infty$. Again by Lemma \ref{decalg4}, $\|f_k\|_\infty\leq C(\eta)|\C_k|^{-1/2}$. If $\eta(R)$ is chosen to tend to zero sufficiently slowly as $R\rightarrow\infty$ to ensure that $2C(\eta(R))\cosh\lambda(\eta(R))<R$ for all $k\leq N$, then 
$$\|f_k\|_\infty<R(2\cosh\lambda(\eta(R)))^{-1}r_k^{-1/2}\leq Rr_0^{-1/2}$$
since $d(\C_k,\C_0)\leq\lambda(\eta(R))$.
It follows that
\begin{equation}\label{req2}
\int_{\{s: F(\gamma(s))\geq R r_0^{-1/2}\}}F(\gamma(s))^2ds\leq\Theta(R),\;\; \forall R\in[1,\bar{R}],
\end{equation}
provided that $\Theta$ is defined as above.

The final function $\eta$ must be chosen to tend to zero slowly enough to satisfy the requirements of the proofs of both \eqref{req1} and \eqref{req2}.
\end{proof}

\section{A concentration compactness result}\label{sec:A concentration compactness result}

Let us start by making precise the previously mentioned notions of {\em uniform integrability} and {\em concentration at a point}.

\begin{definition}
Let $(X,\mathcal{S},\mu)$ be a measure space and let $p\in[1,\infty)$. A subset $\mathcal{U}$ of $L^p(X)$ is called {\em uniformly integrable} of order $p$ if for every $\epsilon>0$ there exists $\delta>0$ such that for every measurable subset $A$ of $X$ for which $\mu(A)<\delta$,
$$\int_A |f|^p d\mu<\epsilon,\;\;\textrm{ for every } f\in\mathcal{U}.$$
\end{definition}

If $\mathcal{U}$ is a bounded subset of $L^p(X)$, it is straightforward to check that $\mathcal{U}$ is uniformly integrable of order $p$ if and only if
$$\lim_{R\rightarrow \infty}\int_{\{|f|>R\}} |f|^p d\mu=0$$
uniformly with respect to $f\in\mathcal{U}$.

If the measure space is finite, then it is well-known that uniform integrability coupled with a weaker form of convergence is enough to ensure strong convergence: 

\begin{proposition}\label{uniform integrability+measure}
Suppose $\mu(X)<\infty$, and let $p\in [1,\infty)$. Let $\{f_n\}$ be a sequence in $L^p(X)$ and let $f\in L^p(X)$. The sequence $\{f_n\}$ converges to $f$ in $L^p$ if (and only if) the following two conditions are satisfied:
\begin{itemize}
\item[(i)] The sequence $\{f_n\}$ converges in measure to $f$;
\item[(ii)] The family $\{f_n: n\in\N\}$ is uniformly integrable of order $p$. 
\end{itemize}
\end{proposition}
\begin{proof} 
We prove only the direction that will be of use to us (the {\em if} part).
The assumptions, together with the fact that any family consisting of one single function is  automatically uniformly integrable, make it clear that the family $\{f_n-f: n\in\N\}$ is also uniformly integrable of order $p$. Given $\epsilon>0$,
$$\int_{\{|f-f_n|<\epsilon\}}|f-f_n|^p d\mu\leq\epsilon^p\mu(X).$$
On the other hand, $\mu(\{|f-f_n|\geq\epsilon\})\rightarrow 0$ as $n\rightarrow\infty$ because of the convergence in measure, and so
$$\lim_{n\rightarrow\infty}\int_{\{|f-f_n|\geq\epsilon\}}|f-f_n|^pd\mu=0$$
by definition of uniform integrability. The conclusion follows.
\end{proof}

\begin{definition}
For $p=\gamma(s_0)\in\Gamma$, we say that a sequence $\{f_n\}$ of functions in $L^2(\sigma)$ satisfying $\|f_n\|_2\rightarrow 1$ as $n\rightarrow \infty$ {\em concentrates at $p$} if for every $\epsilon, r>0$ there exists $N\in\N$ such that, for every $n\geq N$,
$$\int_{|s-s_0|\geq r} |f_n(\gamma(s))|^2 ds<\epsilon.$$
\end{definition}

These introductory remarks are relevant in the context of the following concentration-compactness result which is a consequence of Proposition \ref{upperbounds}:

\begin{proposition}\label{concentrationcompactness}
Let $\{f_n\}$ be an extremizing sequence for \eqref{TS} of nonnegative functions in $L^2(\sigma)$. Then there exists a subsequence, again denoted $\{f_n\}$, and a decomposition $f_n=F_n+G_n$ where $F_n$ and $G_n$ are nonnegative with disjoint supports, $\|G_n\|_2 \rightarrow 0$, and $\{F_n\}$ satisfies one of the two possibilities:
\begin{align}
& \{F_n: n\in\N\}\textrm{ is uniformly integrable of order }2\label{uniformly integrable}.\\
&\{F_n\} \textrm{ concentrates at a point of }\Gamma.\label{concentrates}
\end{align}
\end{proposition}

\begin{proof}
Apply Proposition \ref{upperbounds} to each element of the sequence $\{f_n\}$ to get a decomposition $f_n=F_n+G_n$ where $F_n,G_n\geq 0$ have disjoint supports and $\|G_n\|_2\rightarrow 0$. For each $n$, there exists a cap $\C_n=\C(\gamma(s_n),r_n)$ such that $F_n$ is upper normalized with respect to $\C_n$. It is important to note that Proposition \ref{upperbounds}  yields a uniform statement in $n$ i.e. the gauge function $\Theta$ in conditions \eqref{un2} and \eqref{un3} can be chosen independently of $n$.
Let $r^*:=\lim\sup_{n\rightarrow\infty} r_n$. 

If $r^*>0$, let $\{r_{n_k}\}_k$ be a subsequence  which converges to $r^*$. Renaming, we may assume that $\lim_{n\rightarrow\infty} r_n=r^*$. 
Choosing $N\in\N$ large enough so that $n\geq N$ implies $r_n\geq r^*/4$, we have 
$$\int_{\{F_n>R\}} F_n(\gamma(s))^2ds\leq\int_{\{F_n\geq\frac{R\sqrt{r^*}}{2}r_n^{-1/2}\}}F_n(\gamma(s))^2ds\leq\Theta\Big(\frac{R\sqrt{r^*}}{2}\Big),$$
which tends to 0 as $R\rightarrow\infty$, uniformly in $n$. In other words, the sequence $\{F_n\}$ is uniformly integrable of order 2.
 
If $r^*=0$, choose a subsequence $\{s_{n_k}\}_k$ converging to some $s^*\in [0,\ell]$. Renaming, we may assume that $\lim_{n\rightarrow\infty} s_n=s^*$. 
Let $\epsilon,r>0$ be given. Start by choosing $N_1=N_1(r)$ such that $|s_n-s^*|\leq {r}/{2}$ if $n>N_1$. Then $|s-s^*|\geq r$ implies $|s-s_n|\geq r/2$ if $n>N_1$. Choose $R=R(\epsilon)$ such that $\Theta(R)<\epsilon$. Finally, choose $N_2=N_2(\epsilon,r)$ such that $$r_n\leq\frac{r}{2R}\textrm{ if }n>N_2.$$ 
Such $N_2$ exists since $\lim_{n\rightarrow\infty} r_n=0$. If $n>\max\{N_1,N_2\}$, we then have that
$$\int_{|s-s^*|\geq r} F_n(\gamma(s))^2 ds\leq\int_{|s-s_n|\geq \frac{r}{2}}F_n(\gamma(s))^2 ds\leq\int_{|s-s_n|\geq R r_n}F_n(\gamma(s))^2 ds\leq\Theta(R)<\epsilon$$ i.e. the sequence $\{F_n\}$ concentrates at $\gamma(s^*)$.
\end{proof}

Fanelli, Vega and Visciglia \cite{FVV1} proved the following interesting modification of a well-known result of Br\'{e}zis and Lieb \cite{BL} which does not require the a.e. pointwise convergence of the sequence of functions $\{h_n\}$.

\begin{proposition}[\cite{FVV1}]\label{FVV}
Let $\H$ be a Hilbert space and $T$ be a bounded linear operator from $\H$ to $L^p(\R^d)$, for some $p\in (2,\infty)$. Let $\{h_n\}\in\H$ be such that
\begin{itemize}
\item[(i)] $\|h_n\|_{\H}=1$;
\item[(ii)] $\lim_{n\rightarrow\infty} \|Th_n\|_{L^p(\R^d)}=\|T\|$;
\item[(iii)] $h_n\rightharpoonup h\neq 0$;
\item[(iv)] $T h_n\rightarrow Th$ a.e. in $\R^d$.
\end{itemize}
Then $h_n\rightarrow h$ in $\H$; in particular, $\|h\|_{\H}=1$ and $\|T h\|_{L^p(\R^d)}=\|T\|$.
\end{proposition}

We will be applying this proposition to the adjoint Fourier restriction operator on $\Gamma$ with
$\H=L^2(\sigma)$. We lose no generality in assuming that conditions (i) and (ii) are automatically satisfied by any extremizing sequence $\{f_n\}$. After passing to a subsequence, we may assume that $\{f_n\}$ converges weakly in $L^2(\sigma)$ by Alaoglu's theorem. If $f_n{\rightharpoonup} f$ in $L^2(\sigma)$, then condition (iv) follows because $\sigma$ is compactly supported. Thus Proposition \ref{FVV} states that, for compactly supported measures, the only obstruction to the existence of extremizers is the possibility that every $L^2$ weak limit of any extremizing sequence be zero.

The advantage of working with {\em nonnegative} extremizing sequences in this context first appeared in the work of Kunze \cite{K}.
The following is the sole step in the analysis which works only for nonnegative extremizing sequences:
\begin{lemma}\label{uiimpliesprecompact}
Let $\{f_n\}$ and $\{F_n\}$ be as in Proposition \ref{concentrationcompactness}. Suppose that $\{F_n\}$ satisfies condition \eqref{uniformly integrable}. Then $\{f_n\}$ is precompact in $L^2(\sigma)$. 
\end{lemma}

\begin{proof}
By assumption the sequence $\{F_n\}$ consists of nonnegative functions and is uniformly integrable of order 2.   
Moreover, $\|F_n\|_2\rightarrow 1$ as $n\rightarrow\infty$. 

We first show that every $L^2$ weak limit of $\{F_n\}$ is nonzero. The set of $L^2$ weak limits of $\{F_n\}$ is clearly nonempty. We can assume, possibly after extraction of a subsequence, that $F_n\rightharpoonup F$ for some $F\in L^2(\sigma)$. Suppose that $F=0$ a.e. on $\Gamma$. Then
$$\int_\Gamma F_n d\sigma\rightarrow\int_\Gamma Fd\sigma=0\textrm{ as } n\rightarrow\infty. $$
Since $F_n\geq 0$, this means that the sequence $\{F_n\}$ converges to 0 in $L^1(\sigma)$, and thus $F_n\rightarrow 0$ in measure. In view of Proposition \ref{uniform integrability+measure}, $F_n\rightarrow 0$ in $L^2(\sigma)$, and so $1=\|F_n\|_2\rightarrow 0$ as $n\rightarrow\infty$, a contradiction. Thus $F\neq 0$ as was to be shown.

We use this to prove that the sequence $\{f_n\}$ is precompact.
Since $f_n=F_n+G_n$, we have that
$$\|\widehat{F_n\sigma}\|_6\geq\|\widehat{f_n\sigma}\|_6-\|\widehat{G_n\sigma}\|_6\geq\|\widehat{f_n\sigma}\|_6-{\mathbf C}[\Gamma]\|G_n\|_2.$$
It follows that $\|\widehat{F_n\sigma}\|_6\rightarrow {\mathbf C}[\Gamma]$ as $n\rightarrow\infty$. This means that $\{F_n\}$ is itself an extremizing sequence. By the previous paragraph, we can assume, possibly after extraction of a subsequence, that $F_n\rightharpoonup F$ for some nonzero $F\in L^2(\sigma)$.
Then an application of Proposition \ref{FVV} 
(with $d=2$, $p=6$, $\H=L^2(\sigma)$ and $T=$ Fourier extension operator on $\Gamma$ defined by $Tf:=\widehat{f\sigma}$)
allows us to conclude that $f_n\rightarrow F$ in $L^2(\sigma)$ as $n\rightarrow\infty$, and so $\{f_n\}$ is precompact.
\end{proof}

We will be done with the proof of Theorem \ref{main} once we show that condition \eqref{concentrates} in Proposition \ref{concentrationcompactness} cannot happen, and this is the subject of the next two sections. 

\section{Exploring concentration}\label{sec:Exploring concentration}

 We start by recalling some aspects of Foschi's work \cite{F}. 
Consider the parabola
$$\P:=\{(y,z)\in\R^2: z= y^2\}$$ 
equipped with projection measure\footnote{See \cite{CS} and the references therein for a discussion of why this measure is natural from a geometric point of view.} $d\sigma_{\P}:=dy$ instead of arclength measure. We have an inequality
\begin{equation}\label{foschi}
\|\widehat{f\sigma_{\P}}\|_{L^6(\R^2)}\leq {\mathbf C_F}\|f\|_{L^2(\sigma_{\P})},
\end{equation}
where ${\mathbf C_F}$ denotes again the optimal constant.
Foschi showed that extremizers exist for the inequality \eqref{foschi} and computed the optimal constant 
\begin{equation}\label{Foschi}
{\mathbf C_F}=\frac{(2\pi)^{1/2}}{{12}^{1/12}}.
\end{equation} 
An example of one such extremizer is given by the Gaussian $G(y):=e^{-y^2}$. Other extremizers are obtained from $G$ by space-time translations, parabolic dilations, space rotations, phase shifts and Galilean transformations. 

A straightforward scaling argument shows the following: consider the dilated parabola
$$\P_\mu:=\Big\{(y,z)\in\R^2: z=\frac{\mu y^2}{2}\Big\},$$
again equipped with projection measure $d\sigma_{\P_\mu}=dy$. Then the optimal constant in the inequality
\begin{equation}\label{dilatedfoschi}
\|\widehat{f\sigma_{\P_\mu}}\|_{L^6(\R^2)}\leq {\mathbf C_F}[\mu]\|f\|_{L^2(\sigma_{\P_\mu})}
\end{equation}
satisfies ${\mathbf C_F}[\mu]={\mathbf C_F}[1] \mu^{-1/6}$.
In particular, ${\mathbf C_F}[1]={(2\pi)^{1/2}}{{3}^{-1/12}}$.

Since projection measure can be regarded as a limit of arclength measures, the analysis of extremizers for inequality \eqref{dilatedfoschi} is of significance for our discussion. If an extremizing sequence $\{f_n\}$ for inequality \eqref{TS} concentrates at a point $\gamma(s)\in\Gamma$, then the sequence consisting of certain natural transplantations of $f_n$ to functions $\tilde{f_n}$ (each $\tilde{f_n}$ begin defined on the limiting parabola) will also be an extremizing sequence for \eqref{dilatedfoschi} with $\mu=\kappa(\gamma(s))$. To see why this is the case, denote by $T_{s,r}$ the restriction of the Fourier extension operator to a given cap $\C=\C(s,r)\subset\Gamma$:
$$T_{s,r}f(x,t)=\int_\C f(y)e^{-i(x,t)\cdot y}d\sigma(y)=\int_{s-r}^{s+r}f(\gamma(s))e^{-i(x,t)\cdot\gamma(s)}ds.$$
We are interested in the operator norm $\|T_{s,r}\|:=\sup_{0\neq f\in L^2(\sigma)}\|T_{s,r} f\|_6 \|f\|_{L^2(\sigma)}^{-1}$, and prove the following result:

\begin{proposition}\label{limoc}
For every $s\in(0,\ell)$,
$$\lim_{r\rightarrow 0^+} \|T_{s,r}\|={\mathbf C_F}[\kappa(\gamma(s))].$$
\end{proposition}

\begin{proof}
Fix $s\in(0,\ell)$ and let $\kappa:=\kappa(\gamma(s))$. That the lefthand side is greater than or equal to the righthand side can be easily seen by taking the function $G(y)=e^{-\frac{\kappa y^2}{2}}$ and considering the dilated family $G_\delta(y)=\delta^{-1/2}G(\delta^{-1}y)$ for $\delta>0$. For details, see \S \ref{sec:A variational calculation} below.

So we focus on proving the reverse inequality. Let $\sigma_{s,r}$ denote the restriction of  arclength measure $\sigma$ to the cap $\C(s,r)$, and denote the triple convolution of $\sigma_{s,r}$ with itself by $\sigma_{s,r}^{(\ast 3)}:=\sigma_{s,r}\ast\sigma_{s,r}\ast\sigma_{s,r}$. We have that

\begin{align}
\|T_{s,r}f\|_6^6&=(2\pi)^2\iint_{\R^2} |f\sigma_{s,r}\ast f\sigma_{s,r}\ast f\sigma_{s,r}(\xi,\tau)|^2d\xi d\tau\label{upperboundforopnormT}\notag\\
&\leq(2\pi)^2\iint |f|^2\sigma_{s,r}\ast |f|^2\sigma_{s,r}\ast |f|^2\sigma_{s,r}(\xi,\tau)\cdot\sigma_{s,r}\ast \sigma_{s,r}\ast \sigma_{s,r}(\xi,\tau)d\xi d\tau\notag\\
&\leq(2\pi)^2\sup_{(\xi,\tau)\in\textrm{ supp}(\sigma_{s,r}^{(\ast 3)})}\sigma_{s,r}^{(\ast 3)} (\xi,\tau)\cdot\iint |f|^2\sigma_{s,r}\ast |f|^2\sigma_{s,r}\ast |f|^2\sigma_{s,r}(\xi,\tau)d\xi d\tau\notag\\
&=(2\pi)^2\|\sigma_{s,r}^{(\ast 3)}\|_{L^\infty(\R^2)} \|f\|_{L^2(\sigma_{s,r})}^6,\notag
\end{align}
where we used H\"{o}lder's inequality twice. 
It will then be enough to show that 
$$\|\sigma_{s,r}\ast\sigma_{s,r}\ast\sigma_{s,r}\|_\infty\rightarrow \frac{{\mathbf C_F}[\kappa]^6}{(2\pi)^{2}}\textrm{ as }r\rightarrow 0^+.$$ 

\noindent After applying a rigid motion\footnote{For a more detailed discussion of this procedure, see \S \ref{sec:Introducing local coordinates} below.}  of $\R^2$ to a cap $\C=\C(s,r)$, we may parametrize it in the following way:
\begin{displaymath}
\begin{array}{rcl}
\widetilde{\gamma}_{s,r}: I_r & \rightarrow & \mathbb{R}^2\\
y & \mapsto &  \Big(y,g(y)=\frac{\kappa}{2}y^2+\phi(y)\Big),
\end{array}
\end{displaymath}
where $I_r$ is an interval centered at the origin of length $\asymp r$, $\kappa=g''(0)(1+g'(0)^2)^{-3/2}$ is the curvature of $\Gamma$ at $\gamma(s)$, and $\phi$ is a real-valued smooth function satisfying $\phi(y)=O(|y|^3)$ as $|y|\rightarrow 0$.
We also let $\eta_r\in C_0^\infty(\R)$ be a mollified version of the characteristic function of the interval $I_r$: to accomplish this, fix $\eta\in C_0^\infty(\R)$ such that $\eta\equiv 1$ on $[-1,1]$ and $\eta(\xi)=0$ if $|\xi|\geq 2$, and define $\eta_r:=\eta(2|I_r|^{-1}\cdot)$.

With these definitions we have that, for $|\xi|\leq |I_r|/2$, 
$$\sigma_{s,r}(\xi,\tau)=G_r(\xi)\delta(\tau-g(\xi))d\xi d\tau,$$
where $G_r(\xi):=(1+g'(\xi)^2)^{1/2}\eta_r(\xi)$ is a smooth function supported on $\{\xi\in\R: |\xi|\lesssim r\}$. Observe that $G_r$ (and therefore $\sigma_{s,r}$) depends also on $s$, even though this is not explicitly indicated by the notation, with uniform bounds on $s$ and appropriately uniform dependence on $r$ after dilations.
Following \cite{F} and \cite{Q2} we compute:

\begin{align*}
\sigma_{s,r}^{(\ast 3)} (\xi,\tau)=&\iint G_r(\omega_1)G_r(\omega_2-\omega_1)G_r(\xi-\omega_2)\delta(\tau-g(\omega_1)-g(\omega_2-\omega_1)-g(\xi-\omega_2))d\omega_1d\omega_2\\
=&\iiint_{|\omega|\lesssim r} G_r(\omega_1)G_r(\omega_2)G_r(\omega_3)\delta {\tau-g(\omega_1)-g(\omega_2)-g(\omega_3)\choose \xi-\omega_1-\omega_2-\omega_3} d\omega_1d\omega_2d\omega_3.
\end{align*}

\noindent Change variables $\omega=O\cdot\zeta$, where $O\in SO(3)$ is the orthogonal matrix
\begin{displaymath}
O=
\left(
\begin{array}{ccc}
\frac{1}{\sqrt{3}} & -\frac{1}{\sqrt{2}}  & -\frac{1}{\sqrt{6}} \\
\frac{1}{\sqrt{3}}  &  \frac{1}{\sqrt{2}} & -\frac{1}{\sqrt{6}} \\
\frac{1}{\sqrt{3}}  & 0 & \frac{2}{\sqrt{6}} 
\end{array}\right).
\end{displaymath}

\noindent Under this transformation,
$$\langle(1,1,1),\omega\rangle=\langle(1,1,1), O\cdot\zeta\rangle=\langle O^*\cdot(1,1,1),\zeta\rangle=\sqrt{3}\langle(1,0,0),\zeta\rangle=\sqrt{3}\zeta_1.$$

\noindent The integral becomes 
\begin{multline*}
\sigma_{s,r}^{(\ast 3)} (\xi,\tau)=\iiint_{|\zeta|\lesssim r} G_r(\frac{\zeta_1}{\sqrt{3}}-\frac{\zeta_2}{\sqrt{2}}-\frac{\zeta_3}{\sqrt{6}})G_r(\frac{\zeta_1}{\sqrt{3}}+\frac{\zeta_2}{\sqrt{2}}-\frac{\zeta_3}{\sqrt{6}})G_r(\frac{\zeta_1}{\sqrt{3}}+\frac{2\zeta_3}{\sqrt{6}})\cdot\\
{\cdot\delta {\tau-g(\frac{\zeta_1}{\sqrt{3}}-\frac{\zeta_2}{\sqrt{2}}-\frac{\zeta_3}{\sqrt{6}})-g(\frac{\zeta_1}{\sqrt{3}}+\frac{\zeta_2}{\sqrt{2}}-\frac{\zeta_3}{\sqrt{6}})-g(\frac{\zeta_1}{\sqrt{3}}+\frac{2\zeta_3}{\sqrt{6}})\choose \xi-\sqrt{3}\zeta_1} d\zeta_1d\zeta_2d\zeta_3.}
\end{multline*}

\noindent Renaming variables and setting
$$\mathfrak{G}_r(\xi;\zeta_2,\zeta_3)=G_r\Big(\frac{\xi}{3}-\frac{\zeta_2}{\sqrt{2}}-\frac{\zeta_3}{\sqrt{6}}\Big)G_r\Big(\frac{\xi}{3}+\frac{\zeta_2}{\sqrt{2}}-\frac{\zeta_3}{\sqrt{6}}\Big)G_r\Big(\frac{\xi}{3}+\frac{2\zeta_3}{\sqrt{6}}\Big),$$
this simplifies to 

\begin{multline*}
\sigma_{s,r}^{(\ast 3)} (\xi,\tau)=\frac{\chi(|\xi|\lesssim r)}{\sqrt{3}}\iint_{|(\zeta_2,\zeta_3)|\lesssim r}\mathfrak{G}_r(\xi;\zeta_2,\zeta_3)\cdot\\
\cdot\delta\Big(\tau-\frac{\kappa}{6}\xi^2-\frac{\kappa}{2}|\zeta|^2-\phi\Big(\frac{\xi}{3}-\frac{\zeta_2}{\sqrt{2}}-\frac{\zeta_3}{\sqrt{6}}\Big)-\phi\Big(\frac{\xi}{3}+\frac{\zeta_2}{\sqrt{2}}-\frac{\zeta_3}{\sqrt{6}}\Big)-\phi\Big(\frac{\xi}{3}+\frac{2\zeta_3}{\sqrt{6}}\Big)\Big)d\zeta_2d\zeta_3.
\end{multline*}

Introducing polar coordinates on the $(\zeta_2,\zeta_3)$-plane,
$$\sigma_{s,r}^{(\ast 3)}  (\xi,\tau)=\frac{\chi(|\xi|\lesssim r)}{\sqrt{3}}\int_0^{2\pi}\int_{0\leq\rho\lesssim r}\widetilde{\mathfrak{G}}_r(\xi;\rho,\theta)\delta\Big(\tau-\frac{\kappa}{6}\xi^2-3\phi\Big(\frac{\xi}{3}\Big)-\psi(\xi;\rho,\theta)\Big)\rho d\rho d\theta$$
where
$$\widetilde{\mathfrak{G}}_r(\xi;\rho,\theta)=G_r\Big(\frac{\xi}{3}-\rho\Big(\frac{\cos\theta}{\sqrt{2}}+\frac{\sin\theta}{\sqrt{6}}\Big)\Big)G_r\Big(\frac{\xi}{3}+\rho\Big(\frac{\cos\theta}{\sqrt{2}}-\frac{\sin\theta}{\sqrt{6}}\Big)\Big)G_r\Big(\frac{\xi}{3}+\frac{2}{\sqrt{6}}\rho\sin\theta\Big)$$
and

\begin{multline}
\psi(\xi;\rho,\theta)=\frac{\kappa}{2}\rho^2-3\phi\Big(\frac{\xi}{3}\Big)+\phi\Big(\frac{\xi}{3}-\rho\Big(\frac{\cos\theta}{\sqrt{2}}+\frac{\sin\theta}{\sqrt{6}}\Big)\Big)+\\
\phi\Big(\frac{\xi}{3}+\rho\Big(\frac{\cos\theta}{\sqrt{2}}-\frac{\sin\theta}{\sqrt{6}}\Big)\Big)+\phi\Big(\frac{\xi}{3}+\frac{2}{\sqrt{6}}\rho\sin\theta\Big).
\end{multline}

We prepare to change variables again. Recall the assumption that $\kappa>0$. Note that $\psi(\xi;0,\theta)=0$ and $\psi(\xi;\rho,\theta)>0$ for every sufficiently small $\rho>0$. Moreover, a calculation shows that the same thing happens with first derivatives: $\partial_\rho\psi(\xi;0,\theta)=0$  and $\partial_\rho\psi(\xi;\rho,\theta)>0$ if $\rho>0$ is sufficiently small. We also have that $\partial_\rho^2\psi(\xi;0,\theta)=\kappa+\phi''(\xi/3)=g''(\xi/3)$.

For $u\geq 0$, set $\rho=\rho(u):=\psi^{-1}(u)$, and compute:

\begin{align}
\sigma_{s,r}^{(\ast 3)}  (\xi,\tau)=&\frac{\chi(|\xi|\lesssim r)}{\sqrt{3}}\int_0^{2\pi}\int_{0\leq u\leq C\psi} \widetilde{\mathfrak{G}}_r(\xi;\rho(u),\theta)\delta\Big(\tau-\frac{\kappa}{6}\xi^2-3\phi(\frac{\xi}{3})-u\Big)\frac{\rho(u)}{\partial_\rho\psi(\rho(u))}dud\theta\label{triplesigma1}\\
=&\frac{\chi(|\xi|\lesssim r)}{\sqrt{3}}\int_0^{2\pi}  \widetilde{\mathfrak{G}}_r(\xi;\rho(\tau-\frac{\kappa}{6}\xi^2-3\phi(\frac{\xi}{3})),\theta) \frac{\rho(\tau-\frac{\kappa}{6}\xi^2-3\phi(\frac{\xi}{3}))}{\partial_\rho\psi(\rho(\tau-\frac{\kappa}{6}\xi^2-3\phi(\frac{\xi}{3})))}d\theta.\label{triplesigma2}
\end{align}

\noindent Note that, for each $\theta\in [0,2\pi]$, the integrand in \eqref{triplesigma2} is supported in the region

$$\Big\{(\xi,\tau)\in\R^2: 0\leq\tau-\frac{\kappa}{6}\xi^2-3\phi(\frac{\xi}{3})\leq C\psi(\xi;r,\theta)\Big\},$$ 
where the constant $C<\infty$ is  large enough that the restriction $u\leq C\psi(\xi;r,\theta)$ in the inner integral of \eqref{triplesigma1} becomes redundant because of support limitations on factors present in its integrand.
 From  expression \eqref{triplesigma2} it is clear that the restriction of $\sigma_{s,r}\ast\sigma_{s,r}\ast\sigma_{s,r} $ to its support defines a continuous function of $(\xi,\tau)$ at $(0,0)$. Indeed, $\widetilde{\mathfrak{G}}_r$ is a smooth function of compact support in the variables $\xi,\theta$ and $({\tau-\frac{\kappa}{6}\xi^2-3\phi(\frac{\xi}{3})})^{1/2}$. Additionally, as $(\xi,\tau)\rightarrow(0,0)$,
 
$$\frac{\rho(\tau-\frac{\kappa}{6}\xi^2-3\phi(\frac{\xi}{3}))}{\partial_\rho\psi(\rho(\tau-\frac{\kappa}{6}\xi^2-3\phi(\frac{\xi}{3})))}=\frac{\rho(\tau-\frac{\kappa}{6}\xi^2-3\phi(\frac{\xi}{3}))-\rho(0)}{\partial_\rho\psi(\rho(\tau-\frac{\kappa}{6}\xi^2-3\phi(\frac{\xi}{3})))-\partial_\rho\psi(\rho(0))}\rightarrow\frac{1}{\partial_\rho^2\psi(\rho(0))}=\frac{1}{g''(0)}.$$

\noindent If $(\xi,\tau)\in$ supp$(\sigma_{r,s}\ast\sigma_{r,s}\ast\sigma_{r,s} )$ and $r\rightarrow 0^+$, then $(\xi,\tau)\rightarrow(0,0)$. It follows that

\begin{align*}
\lim_{r\rightarrow 0^+}\|\sigma_{r,s}\ast\sigma_{r,s}\ast\sigma_{r,s}\|_\infty&=\frac{1}{\sqrt{3}}\int_0^{2\pi} \widetilde{\mathfrak{G}}_0(0;\rho(0),\theta)\frac{1}{\partial_\rho^2\psi(0;\rho(0),\theta)}d\theta.\\
&=\frac{2\pi}{\sqrt{3}}\frac{(1+g'(0)^2)^{3/2}}{g''(0)}=\frac{2\pi}{\sqrt{3}}\frac{1}{\kappa}
=\frac{{\mathbf C_F}[\kappa]^6}{(2\pi)^{2}},
\end{align*}
as desired.
\end{proof}

\noindent An immediate consequence is that an extremizing sequence which concentrates must do so at a point of minimal curvature:

\begin{corollary}\label{fconcmincurv}
Let $\{f_n\}\subset L^2(\sigma)$ be an extremizing sequence of nonnegative functions for inequality \eqref{TS}. Suppose that $\{f_n\}$ concentrates at a point $\gamma(s)\in\Gamma$. Then $\kappa(\gamma(s))=\lambda$. 
\end{corollary}

\begin{corollary}\label{Fconcmincurv}
Let $\{f_n\}$ and $\{F_n\}$ be as in Proposition \ref{concentrationcompactness}. Suppose that $\{F_n\}$ concentrates at a point $\gamma(s)\in\Gamma$. Then $\kappa(\gamma(s))=\lambda$. 
\end{corollary}

\section{Comparing optimal constants}\label{sec:Comparing optimal constants}

As was mentioned before, a potential obstruction to the existence of extremizers for  inequality \eqref{TS}, and certainly to the precompactness of arbitrary nonnegative extremizing sequences, is the possibility that for an extremizing sequence satisfying $\|f_n\|_{L^2(\sigma)}=1$, $|f_n|^2$ could conceivably converge weakly to a Dirac mass at a point on the curve. Indeed, let $p\in\Gamma$. The osculating parabola of $\Gamma$ at $p$ is $\P_{\kappa(p)}$. If ${\mathbf C}[\Gamma]$ were equal to ${\mathbf C_F}[\kappa(p)] $, then Foschi's work implies that there would necessarily exist extremizing sequences of the type just described. Therefore an essential step in our analysis is to determine under which conditions one has that 
$${\mathbf C}[\Gamma]>\max_{p\in\Gamma}{\mathbf C_F}[\kappa(p)]
={\mathbf C_F}[\lambda].$$

\noindent The main goal of this section is to prove the following:

\begin{proposition}\label{optimalconstants}
Let $\Gamma\subset\R^2$ be an arc satisfying the conditions of Theorem \ref{main}. 
Then ${\mathbf C}[\Gamma]>{\mathbf C_F}[\lambda]$.
\end{proposition}

\subsection{Introducing local coordinates}\label{sec:Introducing local coordinates}
Let $p_0\in\Gamma$ be a point of minimum curvature i.e. such that $\kappa(p_0)=\lambda$. We will be assuming that $p_0$ is {\em not} an endpoint of $\Gamma$, and we postpone the discussion of the validity of this assumption until the end of this section. By translating the curve we can assume without loss of generality that $p_0=(0,0)$. Possibly after a suitable rotation, the arc $\Gamma$ can be parametrized in a neighborhood of the origin in the following way:

\begin{displaymath}
\begin{array}{rcl}
\widetilde{\gamma}: I & \rightarrow & \mathbb{R}^2\\
y & \mapsto &  (y,h(y)),
\end{array}
\end{displaymath}
where\footnote{There is no cubic term in the expression for $h$ because by assumption the curvature has a minimum at $\gamma(0)=(0,0)=p_0$. Constant and linear terms were likewise removed via the affine change of variables described above.} 
$$h(y)=\frac{ \lambda y^2}{2}+a y^4+\psi(y)$$
and $\psi$ is a real-valued smooth function satisfying $\psi(y)=O(|y|^5)$ as $|y|\rightarrow 0$. 
The parameter $a$, on the other hand, is a function of the second derivative of the curvature with respect to arclength at $0$; see  formula \eqref{2nd} below. We take $I\subseteq\R$ to be an interval centered at the origin which  will be chosen as a function of $\lambda, a$ and $\psi$ later on. Finally, let $\eta_I\in C_0^\infty(\R)$ be a mollified version of the characteristic function of $I$ such that  $\eta_I\equiv 1$ on $I$ and $\eta_I\equiv 0$ outside $2\cdot I$. 
As before, we accomplish this by fixing $\eta\in C_0^\infty(\R)$ such that $\eta\equiv 1$ on $[-1,1]$ and $\eta(y)=0$ if $|y|\geq 2$, and defining $\eta_I:=\eta({2}|I|^{-1}\cdot)$.

These data determine a compact arc $\widetilde{\gamma}(I)=:\widetilde{\Gamma}\subset \Gamma$ in the plane,
which comes equipped with arclength measure $\widetilde{\sigma}$ given, for $y\in I$, by 
 \begin{equation}\label{meas}
d\widetilde{\sigma}(y)=(1+h'(y)^2)^{1/2}\eta_I(y)dy.
\end{equation}

Taylor expanding around $y=0$, we note that the first two  non-zero terms in the expansion of $d\widetilde{\sigma}(y)$ are independent of $a$ and $\psi$:
\begin{equation}\label{apmeas}
d\widetilde{\sigma}(y)=\Big(1+\frac{\lambda^2}{2}y^2+O(y^4)\Big)\eta_I(y)dy.
\end{equation}

On the other hand, the curvature of $\widetilde{\Gamma}$ at a point $\widetilde{\gamma}(y)$ is given, for small $y$, by
$$\kappa(y)=\frac{h''(y)}{(1+h'(y)^2)^{3/2}}=\lambda+(12a-\frac{3\lambda^3}{2})y^2+O_{\lambda, a,\psi}(y^4).$$
We have that $\kappa(0)=h''(0)=\lambda>0$. For $\kappa$ to have a minimum at $y=0$ it is necessary that
$a\geq(\frac{\lambda}{2})^3$.
Let $s$ denote, as usual, the arclength parameter for $\widetilde{\Gamma}$. Then, by a straightforward application of the chain rule, we have that
\begin{equation}\label{2nd}
\frac{d^2\kappa}{ds^2}(0)=\frac{d^2 \kappa}{dy^2}(0)=24a-3\lambda^3,
\end{equation}
and so hypothesis \eqref{k2} in Theorem \ref{main} is  equivalent to $a<\frac{3}{2}(\frac{\lambda}{2})^3$. 
All in all we have that
\begin{equation}\label{a}
\Big(\frac{\lambda}{2}\Big)^3\leq a<\frac{3}{2}\Big(\frac{\lambda}{2}\Big)^3.
\end{equation}

In what follows, we will again denote by ${\mathbf C}[\widetilde{\Gamma}]={\mathbf C}[\widetilde{\Gamma};\lambda,a,\psi,I]$ the optimal constant in the inequality

\begin{equation}\label{TS2}
\|\widehat{f\widetilde{\sigma}}\|_{L^6(\R^2)}\leq {\mathbf C}[\widetilde{\Gamma}]\|f\|_{L^2(\widetilde{\sigma})}.
\end{equation}
\subsection{The unperturbed case}

If $a=\psi=0$, we are dealing with the unperturbed parabola $\P_\lambda$. As we have mentioned before, the corresponding optimal constant satisfies ${\mathbf C_F}[\lambda]={\mathbf C_F}[1] \lambda^{-1/6}$, and examples of extremizers are given by Gaussian functions  $e^{-\rho y^2/2}$  for $\rho>0$. 

Set $G_0(y):=e^{-\lambda y^2/2}$, and consider functions of the form $f=G_0+\phi$ with $\phi\in L^2(\sigma_{\P_\lambda})$. Consider the corresponding functional
\begin{equation}\label{functional}
\phi\mapsto {\mathbf C_F}[\lambda]^6\Big(\int_\R |G_0+\phi|^2 dy\Big)^3-\iint_{\R^2} |G_1+\phi_1|^6 dxdt\geq 0,
\end{equation}
where $G_1(x,t):=\widehat{G_0\sigma_{\P_\lambda}}(-x,t)$ and $\phi_1(x,t):=\widehat{\phi\sigma_{\P_\lambda}}(-x,t)$. The former can be explicitly computed:

\begin{align*}
G_1(x,t)=&\widehat{G_0\sigma_{\P_\lambda}}(-x,t)=\int_\R G_0(y)e^{-it\frac{\lambda y^2}{2}}e^{ixy}dy\\
=&\int_\R e^{-(1+it)\frac{\lambda y^2}{2}}e^{ixy}dy\\
=&\Big(\frac{2\pi}{\lambda}\Big)^{1/2}(1+it)^{-1/2}e^{-\frac{x^2}{2 \lambda(1+it)}}.
\end{align*}

\noindent One readily checks that $G_1\in L^p_{x,t}(\R^2)$ if and only if $p>4$, but we are interested in $L^6$ norms. Since $\widehat{G_0\sigma_{\P_\lambda}}$ is an even function of $x$, 
\begin{equation}\label{CF}
\|G_1\|_{L^6(\R^2)}={\mathbf C_F}[\lambda]\|G_0\|_{L^2(\sigma_{\P_\lambda})}.
\end{equation}

We follow the work of \cite{DMR} and expand the functional \eqref{functional} up to second order, collecting the terms which do not depend on $\phi$ in ${\mathbf I}$, the ones which depend linearly on the real and imaginary parts of $\phi$ in ${\mathbf{II}}$, and  the ones which depend quadratically on the real and imaginary parts of $\phi$ in ${\mathbf{III}}$. This yields:

\begin{equation}\label{O1}
{\mathbf I}={\mathbf C_F}[\lambda]^6\|G_0\|_2^6-\|G_1\|_6^6;
\end{equation}

\begin{equation}\label{O2}
{\mathbf{II}}=6{\mathbf C_F}[\lambda]^6 \|G_0\|_2^4\Re \int G_0\phi-6\Re \iint |G_1|^4\overline{G_1}\phi_1;
\end{equation}

\begin{multline}\label{O3}
{\mathbf{III}}=3{\mathbf C_F}[\lambda]^6 \|G_0\|_2^4\int|\phi|^2+12{\mathbf C_F}[\lambda]^6 \|G_0\|_2^2\Big(\Re\int G_0\phi\Big)^2\\
-9\iint|G_1|^4|\phi_1|^2
-6\Re\iint |G_1|^2\overline{G_1}^2\phi_1^2.
\end{multline}

We already know from \eqref{CF} that ${\mathbf I}=0$. Since $G_0$ is an extremizer, we have that ${\mathbf{II}}=0$ as well. Finally, note that ${\mathbf{III}}$ is, by definition, a quadratic form in $\phi$; denote it by $Q(\phi)$. 
By the symmetries of the problem (respectively, multiplication by a real number, phase shift, space translation, Galilean invariance, scaling and time translation), we have that 
\begin{equation}\label{ker}
Q(G_0)=Q(iG_0)=Q(yG_0)=Q(iyG_0)=Q(y^2G_0)=Q(iy^2G_0)=0,
\end{equation}
and it is proved in \cite{DMR} that $Q$ is positive definite in the subspace of $L^2$ functions which are orthogonal to the functions indicated in \eqref{ker}.  This non-degeneracy property will not be used here; rather, what is essential for our application is  that  $Q(\phi)\geq 0$ for every $\phi\in L^2(\R)$. This is immediate because  ${\mathbf{II}}$ vanishes for every $\phi\in L^2$ and \eqref{functional} defines a nonnegative quantity.

\subsection{A variational calculation}\label{sec:A variational calculation}

In the spirit of the variational calculation in  \cite[\S 17]{CS}, we consider the one-parameter family of trial functions given by
$$(G_0+\epsilon\varphi)_{0<\epsilon\leq\epsilon_0},$$
for some sufficiently small $\epsilon_0>0$, where $G_0(y)=e^{-\lambda y^2/2}$ and $\varphi\in L^2(\R)$ will be chosen below, in such a way that 
\begin{equation}\label{conditions}
\|\varphi\|_2=1 \textrm{ and } \int G_0\varphi=0.
\end{equation}
For technical reasons that will become apparent soon, we introduce an appropriate dilation  of the  cut-off $\eta_I$ which localizes to the region $|y|\lesssim \epsilon\log\frac{1}{\epsilon}$, and define

\begin{equation}\label{family}
f_\epsilon(y):=\epsilon^{-1/2}(G_0+\epsilon\varphi)(\epsilon^{-1} y)\eta_I\Big(\frac{1}{\epsilon\log\frac{1}{\epsilon}}y\Big).
\end{equation}
Notice that the family $(f_\epsilon)_{\epsilon>0}$ is $L^2$-normalized in the sense that
\begin{equation}\label{normal}
\|f_\epsilon\|_{L^2(\widetilde{\sigma})}^2=\|G_0\|_2^2+O(\epsilon^2)\textrm{ as } \epsilon\rightarrow 0^+.
\end{equation}

Consider the quantity:
\begin{equation}\label{psi}
\Xi(\epsilon)={\mathbf C_F}[\lambda]^6\|f_\epsilon\|_{L^2(\widetilde{\sigma})}^6-\|\widehat{f_\epsilon\widetilde{\sigma}}\|_6^6.
\end{equation}
This is no longer a nonnegative expression by construction like \eqref{functional}. 

Let 
\begin{equation}\label{phi}
\varphi(u):=c_\lambda u e^{-\frac{\lambda u^2}{2}},
\end{equation}
where the constant $c_\lambda:=(\pi/4\lambda^3)^{-1/4}$ is chosen to normalize $\|\varphi\|_2=1$.
Then $\varphi$ satisfies conditions \eqref{conditions}
and $Q(\varphi)=0$.

With this choice of $\varphi$, we claim that the function $\Xi=\Xi(\epsilon)$ has the following property: for every $\lambda>0$, for every $a\in \R$ satisfying \eqref{a}, and for every real-valued smooth $\psi$ satisfying $\psi(y)=O(|y|^5)$ as $|y|\rightarrow 0$, $\Xi$ is a strictly concave function of $\epsilon$ in a sufficiently small half-neighborhood of 0, provided the interval $I$ is chosen sufficiently small (as a function of $\lambda, a$ and $\psi$). Once we prove this we will be able to conclude that ${\mathbf C_F}[\lambda]<{\mathbf C}[\widetilde{\Gamma}]$, and Proposition \ref{optimalconstants}  follows.

Start by noting that $\lim_{\epsilon\rightarrow 0}\Xi(\epsilon)=0$ and $\Xi'(0)=0$. 
 Indeed, one has that 
$$\lim_{\epsilon\rightarrow 0}\Xi(\epsilon)={\mathbf C_F}[\lambda]^6\lim_{\epsilon\rightarrow 0}\|f_\epsilon\|_{L^2(\widetilde{\sigma})}^6-\lim_{\epsilon\rightarrow 0}\|\widehat{f_\epsilon\widetilde{\sigma}}\|_6^6={\mathbf C_F}[\lambda]^6\|G_0\|_2^6-\|G_1\|_6^6=0$$
for every $a\in\R$ and $\lambda>0$. On the other hand, explicit computations show that 
$$\partial_\epsilon|_{\epsilon=0} \|f_\epsilon\|_{L^2(\widetilde{\sigma})}^6=6\|G_0\|_2^4\Re \int G_0\varphi$$
and 
$$\partial_\epsilon|_{\epsilon=0} \|\widehat{f_\epsilon\widetilde{\sigma}}\|_6^6=6\Re\iint |G_1|^4\overline{G_1}\varphi_1.$$
Since  ${\mathbf{II}}=0$ and $G_0 \perp \varphi$, we have that 
\begin{equation}\label{ReI=0}
\Re\iint |G_1|^4\overline{G_1}\varphi_1={\mathbf C_F}[\lambda]\|G_0\|_2^4\Re \int G_0\varphi=0
\end{equation}
and it follows that $\Xi'(0)=0$, as claimed.

Useful information will come from looking at second variations. The strategy will be to compute the second derivatives with respect to $\epsilon$ (at $\epsilon=0$) of good enough approximations to the two terms appearing in the definition of $\Xi$. We start by analyzing the most involved one.

\subsubsection{The term $\|\widehat{f_\epsilon\widetilde{\sigma}}\|_6^6$.}\label{sec:The hard term}

\noindent To make the notation less cumbersome, we introduce the following parametrization:

\begin{displaymath}
\begin{array}{rcl}
{\widetilde{\gamma}}_\epsilon: \epsilon^{-1}\cdot I & \rightarrow & \mathbb{R}^2\\
u & \mapsto &  \Big(u,\widetilde{h}_\epsilon(u)=\frac{\lambda u^2}{2}+a\epsilon^2 u^4+\epsilon^{-2}\psi(\epsilon u)\Big).
\end{array}
\end{displaymath}
Changing variables $y=\epsilon u$, we have that:
\begin{align*}
\widehat{f_\epsilon\widetilde\sigma}(x,t)=&\int_\R f_\epsilon(y)e^{-i(x,t)\cdot(y,h(y))}(1+h'(y)^2)^{1/2}\eta_I(y)dy\\
=&\epsilon^{1/2}\int_\R (G_0+\epsilon\varphi)(u)e^{-i(\epsilon x,\epsilon^2 t)\cdot (u,\widetilde{h}_\epsilon(u))}(1+h'(\epsilon u)^2)^{1/2}\eta_I\Big(\frac{1}{\log\frac{1}{\epsilon}} u\Big)du.
\end{align*}

\noindent Consider an approximate version of
\begin{align*}
w_\epsilon(x,t):=&\epsilon^{-1/2}\widehat{f_\epsilon\widetilde{\sigma}}(\epsilon^{-1}x,\epsilon^{-2}t)\\
=&\int_\R (G_0+\epsilon\varphi)(u)e^{-i(x, t)\cdot (u,\widetilde{h}_\epsilon(u))}(1+h'(\epsilon u)^2)^{1/2}\eta_I\Big(\frac{1}{\log\frac{1}{\epsilon}} u\Big)du
\end{align*}
given by Taylor expanding $J_\epsilon(u):=\sqrt{1+h'(\epsilon u)^2}\in  C^\infty(I)$ and defining
\begin{equation}\label{defvepsil}
v_\epsilon(x,t):=\int_\R (G_0+\epsilon\varphi)(u)e^{-i(x, t)\cdot (u,\widetilde{h}_\epsilon(u))}\Big(1+\frac{\lambda^2}{2} \epsilon^2u^2\Big)\eta_I\Big(\frac{1}{\log\frac{1}{\epsilon}} u\Big)du.
\end{equation}
Also, set
$$g_\epsilon(u):=(G_0+\epsilon\varphi)(u),\;\;\; g^\sharp_\epsilon(u):=u^2(G_0+\epsilon\varphi)(u) \;\textrm{ and }\;d\widetilde{\sigma}_\epsilon(u):=(1+\frac{\lambda^2}{2}\epsilon^2u^2)du.$$

\begin{lemma}\label{propest}
If $I$ is a sufficiently small interval centered at the origin (chosen as a function of $\lambda, a,\psi$ but not $\epsilon$), then
\begin{equation}\label{wf}
\|\widehat{f_\epsilon\widetilde{\sigma}}\|_6^6=\|w_\epsilon\|_6^6=\|v_\epsilon\|_6^6+O(\epsilon^{4})
\end{equation}
and 
\begin{equation}\label{af}
\|f_\epsilon\|_{L^2(\widetilde{\sigma})}^2=\|g_\epsilon\|_{L^2(\widetilde{\sigma}_\epsilon)}^2+O(\epsilon^{4}).
\end{equation}
as $\epsilon\rightarrow 0^+$.
\end{lemma}

\begin{proof}
By construction, $v_\epsilon$ is just $w_\epsilon$ with the term $J_\epsilon=(1+h'(\epsilon\cdot)^2)^{1/2}$ replaced by its Taylor approximation to order 2. For $\epsilon<1$, set

$$G_\epsilon(u):=(1+\frac{\lambda^2}{2} \epsilon^2u^2)(G_0+\epsilon\varphi)(u)$$
and
$$s_\epsilon(x,t):=-\int_\R G_\epsilon(u)e^{-i(x, t)\cdot\widetilde{\gamma}_\epsilon(u)}\Big(\eta_I(\epsilon u)-\eta_I\Big(\frac{1}{\log\frac{1}{\epsilon}} u\Big)\Big)du.$$
It follows that 
\begin{equation}\label{v3sum}
v_\epsilon(x,t)=\widehat{g_{\epsilon}{\sigma_{P,\epsilon}}}(x,t)+\frac{\lambda^2\epsilon^2}{2}\widehat{{g^\sharp_{\epsilon}}{\sigma_{P,\epsilon}}}(x,t)+s_\epsilon(x,t),
\end{equation}
where $d\sigma_{P,\epsilon}=\eta_I(\epsilon u)du$.
We obtain a uniform estimate for its $L^6$ norm:
\begin{claim}\label{e1}
There exists a constant $C<\infty$ such that
 $\|v_\epsilon\|_6\leq C$, for every sufficiently small $\epsilon>0$. 
\end{claim}

\noindent The claimed uniformity in $\epsilon$ needs to be justified: it follows from undoing the substitutions $y\mapsto \epsilon^{-1}y$ and $(x,t)\mapsto(\epsilon^{-1}x,\epsilon^{-2}t)$. Indeed,

\begin{align*}
\widehat{g_{\epsilon}{\sigma_{P,\epsilon}}}(x,t)=&\int g_{\epsilon}(u) e^{-i(x,t)\cdot(u,\widetilde{h}_\epsilon(u))}\eta_I(\epsilon u)du\\
=&\epsilon^{-1/2}\int\epsilon^{-1/2}(G_0+\epsilon\varphi)(\epsilon^{-1}y)e^{-(\epsilon^{-1}x,\epsilon^{-2}t)\cdot(y,h(y))}\eta_I(y)dy\\
=&\epsilon^{-1/2} \widehat{f_\epsilon\sigma_P}(\epsilon^{-1}x,\epsilon^{-2}t),
\end{align*}
where $d\sigma_{P}(y):=\eta_I(y)dy$. Since $\|\epsilon^{-1/2} \widehat{f_\epsilon\sigma_P}(\epsilon^{-1}\cdot,\epsilon^{-2}\cdot)\|_6=\| \widehat{f_\epsilon\sigma_P}\|_6$, we have that
$$\|\widehat{g_{\epsilon}\sigma_{P,\epsilon}}\|_6=\|\epsilon^{-1/2} \widehat{f_\epsilon\sigma_P}(\epsilon^{-1}\cdot,\epsilon^{-2}\cdot)\|_6=\| \widehat{f_\epsilon\sigma_P}\|_6\lesssim \|f_\epsilon\|_2\leq C\|G_0\|_2,$$
for some $C<\infty$ independent of $\epsilon$, as claimed. Proceed similarly to get a bound $O(\epsilon^2)$ for the term involving ${g^\sharp_{\epsilon}}$. 
Finally, define 
$$g^\flat_{\epsilon}(u):=-G_\epsilon(u)\Big(\eta_I(\epsilon u)-\eta_I\Big(\frac{1}{\log\frac{1}{\epsilon}} u\Big)\Big)$$ 
and notice that $s_\epsilon(x,t)=\widehat{g^\flat_{\epsilon}\sigma_{P,\epsilon}}(x,t)$. Estimate:

\begin{align*}
\|g^\flat_{\epsilon}\|_2^2&\asymp\int_{\log{\epsilon^{-1}}\lesssim|u|\lesssim\epsilon^{-1}}\Big(1+c_\lambda\epsilon u+\frac{\lambda^2\epsilon^2}{2} u^2+c_\lambda\frac{\lambda^2\epsilon^3}{2} u^3\Big)^2e^{-\lambda u^2}du\\
&\lesssim e^{-C(\log{\epsilon^{-1}})^{2}}\lesssim_{N}\epsilon^N,\textrm{ for every } N\in\N.
\end{align*}
Eliminating the substitutions as before, we conclude that
$$\|s_\epsilon\|_6\lesssim \|g^\flat_{\epsilon}\|_2\lesssim\epsilon^N,\;\;\; \forall N\in\N,$$
where the implicit constants are all independent of $\epsilon$. Thus the contribution of the third summand is likewise small, and this concludes the verification of Claim \ref{e1}.

If we choose the interval $I$ small enough (as a function of $\lambda, a $ and $\psi$) such that 
$$y\in I\Rightarrow |h'(y)|=|\lambda y+4a y^3+\psi'(y)|\leq 1,$$
then the remainder
$$r_\epsilon(x,t):=w_\epsilon(x,t)-v_\epsilon(x,t)$$
will satisfy favorable bounds. By Taylor's theorem we have that
\begin{equation}\label{rbounds}
|r_\epsilon(x,t)|\leq C\epsilon^4 \Big|\int_\R J_\epsilon^{''''}(c_0)u^4(G_0+\epsilon\varphi)(u)e^{-i(x, t)\cdot \widetilde{\gamma_\epsilon}(u)}\eta_I\Big(\frac{1}{\log\frac{1}{\epsilon}} u\Big)du\Big|,
\end{equation}
for some $c_0\in (-\epsilon u,\epsilon u)$ and some absolute constant $C<\infty$.
An argument analogous to the one used to establish Claim \ref{e1} yields the following estimate for the remainder term:

\begin{claim}\label{e2}
There exists a constant $C<\infty$ such that
 $\|r_\epsilon\|_6\leq C\epsilon^4$, for every sufficiently small $\epsilon>0$. 
 \end{claim}

To finish the proof of Lemma \ref{propest},
notice that
$\|w_\epsilon\|_6^6=\|v_\epsilon+r_\epsilon\|_6^6=\|v_\epsilon\|_6^6+63\textrm{  terms,}$
all of which are $O(\epsilon^4)$ as $\epsilon\rightarrow 0^+$. This is an immediate consequence of H\"{o}lder's inequality, together with Claims \ref{e1} and \ref{e2}: for example,
$$\iint|v_\epsilon|^4\overline{v_\epsilon}r_\epsilon dxdt\leq\|v_\epsilon\|_6^5\|r_\epsilon\|_6\leq C\epsilon^4.$$
All other terms can be dealt with in a similar way, and the result follows.
The verification of \eqref{af} is easier and we omit the details.
\end{proof}
Since we are interested in second variations with respect to $\epsilon$ of the $L^6$ norm $\|\widehat{f_\epsilon\widetilde{\sigma}}\|_6^6$ at $\epsilon=0$, it will suffice, in light of \eqref{wf}, to analyze $\|v_\epsilon\|_6^6$.
Start by noting that
 \begin{align}
&v_0(x,t)=G_1(x,t)\label{p1}\\
&\partial_\epsilon|_{\epsilon=0} v_\epsilon(x,t)=\varphi_1(x,t)\label{p2}\\
&\partial_\epsilon^2|_{\epsilon=0} v_\epsilon(x,t)=\lambda^2 G_2(x,t)-2 i t a G_3(x,t)\label{p3}
\end{align}
where, as before, 
$$G_1(x,t):=\widehat{G_0\sigma_{\P_\lambda}}(-x,t)=\Big(\frac{2\pi}{\lambda}\Big)^{1/2}(1+it)^{-1/2}e^{-\frac{x^2}{2 \lambda(1+it)}}$$ 
and 
$$\varphi_1(x,t):=\widehat{\varphi\sigma_{\P_\lambda}}(-x,t)
= \frac{ic_\lambda}{\lambda}\Big(\frac{2\pi}{\lambda}\Big)^{1/2}(1+it)^{-3/2}xe^{-\frac{x^2}{2\lambda(1+it)}}.$$

\noindent Additionally,
\begin{align*}
G_2(x,t)&:=[(u^2G_0)\sigma_{\P_\lambda}]^\wedge(-x,t)\\
=&\int_\R y^2G_0(y)e^{-it\frac{\lambda y^2}{2}}e^{ixy}dy=2 \lambda^{-1} i\partial_t G_1(x,t)\\
=&\Big(\lambda^{-1}(1+it)^{-1}-\lambda^{-2} x^2(1+it)^{-2}\Big)G_1(x,t)
\end{align*}
and 
\begin{align*}
G_3(x,t)&:=[(u^4G_0)\sigma_{\P_\lambda}]^\wedge(-x,t)\\
=&\int_\R y^4G_0(y)e^{-it\frac{\lambda y^2}{2}}e^{ixy}dy=-4 \lambda^{-2}\partial_t^2 G_1(x,t)\\
=&\Big(3\lambda^{-2}(1+it)^{-2}-6\lambda^{-3}x^2(1+it)^{-3}+\lambda^{-4}x^4(1+it)^{-4}\Big) G_1(x,t).
\end{align*}

\begin{remark}
The calculations to follow, which lead to formula \eqref{w} below, are largely formal and need to be justified. In particular, the fact that $\epsilon\mapsto \|v_\epsilon\|_6^6$ is twice differentiable at $\epsilon=0$ is proved in Appendix 1.
\end{remark}

\noindent As a first step in the direction of computing $\partial_\epsilon^2|_{\epsilon=0}\|v_{\epsilon}\|_6^6$, we look at $\partial_\epsilon |v_\epsilon|^6(x,t)$:
\begin{align*}
\partial_\epsilon |v_\epsilon|^6=&\partial_\epsilon(v_\epsilon^3\overline{v_\epsilon}^3)\\
=&3 v_\epsilon^2 (\partial_\epsilon v_\epsilon)\overline{v_\epsilon}^3+3 \overline{v_\epsilon}^2 (\partial_\epsilon \overline{v_\epsilon}) {v_\epsilon}^3.
\end{align*}
Recalling \eqref{p1}$-$\eqref{p2}, it follows that
$$\partial_\epsilon|_{\epsilon=0} |v_\epsilon|^6=6\Re(G_1^2\varphi_1\overline{G_1}^3)$$
and so \eqref{ReI=0} implies 
$$\partial_\epsilon|_{\epsilon=0} \|v_\epsilon\|_6^6=6\Re\iint |G_1|^4\overline{G_1}\varphi_1dxdt=0,$$
which we already knew.
We differentiate once again and obtain
$$\partial_\epsilon^2 |v_\epsilon|^6=2\Re\Big(6 v_\epsilon (\partial_\epsilon v_\epsilon)^2\overline{v_\epsilon}^3+3v_\epsilon^2(\partial_\epsilon^2 v_\epsilon)\overline{v_\epsilon}^3+9 |v_\epsilon|^4 |\partial_\epsilon v_\epsilon|^2\Big),$$
which at $\epsilon=0$ equals (recall \eqref{p1}$-$\eqref{p3})

\begin{align*}
\partial_\epsilon^2|_{\epsilon=0} |v_\epsilon|^6=&2\Re\Big(6 v_0 (\partial_\epsilon|_{\epsilon=0} v_\epsilon)^2\overline{v_0}^3+3v_0^2(\partial_\epsilon^2|_{\epsilon=0} v_\epsilon)\overline{v_0}^3+9 |v_0|^4 |\partial_\epsilon|_{\epsilon=0} v_\epsilon|^2\Big)\\
=&2\Re\Big(6 G_1\varphi_1^2\overline{G_1}^3+3G_1^2(\lambda^2 G_2-2 i t a G_3)\overline{G_1}^3+9|G_1|^4|\varphi_1|^2\Big).
\end{align*}
Thus 

\begin{multline}
\frac{1}{2}\partial_\epsilon^2|_{\epsilon=0} \|v_\epsilon\|_6^6=9 \iint |G_1|^4|\varphi_1|^2dxdt+6\Re \iint |G_1|^2\overline{G_1}^2\varphi_1^2dxdt\label{w}\\
+3\lambda^2\Re\iint |G_1|^4\overline{G_1}G_2dxdt-6a\iint \Re \Big\{it |G_1|^4\overline{G_1}G_3\Big\}dxdt.
\end{multline}

\noindent The first two summands on the right hand side of  \eqref{w} appear (with opposite signs) in the expression \eqref{O3} for the quadratic form $Q$. The last two summands, on the other hand, can be explicitly evaluated, and that is our next task. 
The proofs of the following claims are deferred to  Appendix 2:

\begin{claim}\label{c1}
\begin{equation}\label{claim1}
3\lambda^2\Re\iint_{\R^2} |G_1(x,t)|^4\overline{G_1(x,t)}G_2(x,t) dxdt=\frac{3}{2}\pi^{3/2}\lambda^{-1/2}{\mathbf C_F}[\lambda]^6.
\end{equation}
\end{claim}

\begin{claim}\label{c2}
\begin{equation}\label{claim2}
-6a\iint_{\R^2} \Re\Big\{it|G_1(x,t)|^4\overline{G_1(x,t)}G_3(x,t)\Big\} dxdt=-4a\pi^{3/2}\lambda^{-7/2}{\mathbf C_F}[\lambda]^6.
\end{equation}
\end{claim}

\noindent It follows from \eqref{w}, \eqref{claim1} and \eqref{claim2} that

\begin{multline}
\frac{1}{2}\partial_\epsilon^2|_{\epsilon=0} \|v_\epsilon\|_6^6=  9\iint |G_1|^4|\varphi_1|^2 dxdt+6\Re\iint |G_1|^2\overline{G_1}^2\varphi_1^2 dxdt\label{w2}\\
+\frac{3}{2}\pi^{3/2}\lambda^{-1/2}{\mathbf C_F}[\lambda]^6-4a\pi^{3/2}\lambda^{-7/2}{\mathbf C_F}[\lambda]^6.
\end{multline}

\subsubsection{The term $\|f_\epsilon\|_{L^2(\widetilde{\sigma})}^6$}
In view of \eqref{af}, it will be enough to compute the approximate expression
$\|g_\epsilon\|_{L^2(\widetilde{\sigma}_{\epsilon})}^2.$
Since $G_0\perp\varphi$, we have that 
\begin{align*}
\|g_\epsilon\|_{L^2(\widetilde{\sigma}_\epsilon)}^2=&\int_\R |(G_0+\epsilon\varphi)(u)|^2(1+\frac{\lambda^2}{2}u^2\epsilon^2)du\\
=&\int|G_0(u)|^2du+\Big(\int |\varphi(u)|^2du+ \frac{\lambda^2}{2}\int u^2|G_0(u)|^2du\Big)\epsilon^2\\
&\qquad\qquad\qquad +\Big(\lambda^2\Re\int u^2G_0(u)\varphi(u)du\Big)\epsilon^3+\Big(\frac{\lambda^2}{2}\int u^2|\varphi(u)|^2du\Big)\epsilon^4,
\end{align*}
which in turn implies $\partial_\epsilon|_{\epsilon=0} \|g_\epsilon\|_{L^2(\widetilde{\sigma}_\epsilon)}^2=0$
and
$$\partial_\epsilon^2|_{\epsilon=0} \|g_\epsilon\|_{L^2(\widetilde{\sigma}_\epsilon)}^6=6\|G_0\|_{2}^4\Big(\int |\varphi(u)|^2du+\frac{\lambda^2}{2}\int u^2|G_0(u)|^2du\Big).$$
One last computation shows that
\begin{align*}
6\|G_0\|_2^4\Big(\frac{\lambda^2}{2}\int u^2|G_0(u)|^2du\Big)=&{3\lambda^2}\Big(\int_{-\infty}^\infty e^{-\lambda u^2}du\Big)^2\Big(\int_{-\infty}^\infty u^2e^{-\lambda u^2}du\Big)\\
=& {3\lambda^2}\frac{\pi}{\lambda}\frac{\pi^{1/2}}{2\lambda^{3/2}}  =\frac{3}{2}\pi^{3/2}\lambda^{-1/2},
\end{align*}
and so

\begin{equation}\label{f}
\frac{1}{2}\partial_\epsilon^2|_{\epsilon=0}\|g_\epsilon\|_{L^2(\widetilde{\sigma}_\epsilon)}^6=3\|G_0\|_2^4\int|\varphi(u)|^2du+\frac{3}{4}\pi^{3/2}\lambda^{-1/2}.
\end{equation}

\subsubsection{Putting it all together}
Using the approximations \eqref{wf} and \eqref{af} given by Lemma \ref{propest}, we get that:

\begin{equation*}
\frac{\Xi''(0)}{2}={\mathbf C_F}[\lambda]^6\frac{1}{2}\partial_\epsilon^2|_{\epsilon=0}\|g_\epsilon\|_{L^2(\widetilde{\sigma}_\epsilon)}^6-\frac{1}{2}\partial_\epsilon^2|_{\epsilon=0}\|v_\epsilon\|_6^6.
\end{equation*}
The main terms in these approximations have been explicitly computed in \eqref{w2} and \eqref{f}. We obtain:

\begin{align*}
\frac{\Xi''(0)}{2}=3{\mathbf C_F}[\lambda]^6&\|G_0\|_2^4\int|\varphi(s)|^2ds+\frac{3}{4}{\mathbf C_F}[\lambda]^6\pi^{3/2}\lambda^{-1/2}- 9\iint |G_1|^4|\varphi_1|^2 dxdt\\
&-6\Re\iint |G_1|^2\overline{G_1}^2\varphi_1^2 dxdt-\frac{3}{2}\pi^{3/2}\lambda^{-1/2}{\mathbf C_F}[\lambda]^6
+4a\pi^{3/2}\lambda^{-7/2}{\mathbf C_F}[\lambda]^6.\\
\end{align*}
Recalling the definition \eqref{O3} of the quadratic form $Q$ and the fact that $G_0\perp\varphi$,

\begin{equation*}
\frac{\Xi''(0)}{2}=Q(\varphi)-\frac{3}{4}\pi^{3/2}\lambda^{-1/2}{\mathbf C_F}[\lambda]^6+4a\pi^{3/2}\lambda^{-7/2}{\mathbf C_F}[\lambda]^6.
\end{equation*}
It follows that
$$\Xi(\epsilon)=\Big(Q(\varphi)-\frac{3}{4}\pi^{3/2}\lambda^{-1/2}{\mathbf C_F}[\lambda]^6+4a\pi^{3/2}\lambda^{-7/2}{\mathbf C_F}[\lambda]^6\Big) \epsilon^2+O(\epsilon^{3})$$
for sufficiently small $\epsilon>0$, and so $\Xi$ is strictly concave in a neighborhood of 0 if and only if 
$$Q(\varphi)-\frac{3}{4}\pi^{3/2}\lambda^{-1/2}{\mathbf C_F}[\lambda]^6+4a\pi^{3/2}\lambda^{-7/2}{\mathbf C_F}[\lambda]^6<0,$$
that is, if and only if 
$$a<\frac{3}{2}\Big(\frac{\lambda}{2}\Big)^3-\frac{Q(\varphi)}{4\pi^{3/2}{\mathbf C_F}[\lambda]^6} \lambda^{7/2}.$$
The right hand side of this expression equals $\frac{3}{2}\Big(\frac{\lambda}{2}\Big)^3$ since $Q(\varphi)=0$.

We have proved that, for the choice of $\varphi$ given by \eqref{phi}, for every $\lambda>0$, for every $a\in \R$ satisfying \eqref{a} and for every real-valued smooth $\psi$ satisfying $\psi(y)=O(|y|^5)$ as $|y|\rightarrow 0$, $\Xi$ is a strictly concave function of $\epsilon$, for sufficiently small $\epsilon$. In particular, $\Xi(\epsilon)<0$, which is equivalent to
$${\mathbf C_F}[\lambda]^6\|f_\epsilon\|_{L^2(\widetilde{\sigma})}^6<\|\widehat{f_\epsilon\widetilde{\sigma}}\|_6^6.$$
Together with \eqref{TS2}, this implies
$${\mathbf C_F}[\lambda]<{\mathbf C}[\widetilde{\Gamma}].$$
Trivially, ${\mathbf C}[\widetilde{\Gamma}]\leq {\mathbf C}[\Gamma]$, and this finishes the proof of Proposition \ref{optimalconstants} modulo the work deferred to the Appendices.
 
\begin{remark} 
We have been working under the additional assumption that the curvature $\kappa$ of $\Gamma$ does not attain a global minimum at one of its endpoints. This represents no loss of generality. Indeed, if that were not the case and, say, $\kappa$ attained a global minimum at $p_0=\gamma(0)$, then we could perform an identical variational calculation with functions $f_\epsilon$ defined in a similar way to \eqref{family} but supported instead on small neighborhoods $\widetilde{\Gamma}_\epsilon$ of points $p_\epsilon=\gamma(s_\epsilon)\in\Gamma$ with $|s_\epsilon|\asymp\epsilon$. It is straightforward to check that the  computation carries through. We omit the details.
\end{remark}

\section{The end of the proof}\label{sec:The end of the proof}

Recall that by $T_{s,r}$ we mean the adjoint Fourier restriction operator on a cap $\C=\C(s,r)\subset\Gamma$. The following (easy) estimate is the last one we need in order to finish the argument:

\begin{lemma}\label{limlim}
If a sequence $\{f_n\}\subset L^2(\sigma)$ concentrates at a point $\gamma(s)\in\Gamma$, then
$$\lim_{n\rightarrow\infty}\|\widehat{f_n\sigma}\|_6\leq\lim_{r\rightarrow 0^+}\lim_{n\rightarrow\infty}  \|T_{s,r} f_n\|_6.$$
\end{lemma}

\begin{proof}
Set $Tf:=\widehat{f\sigma}$, and let $r>0$ be arbitrary.  
The Tomas-Stein inequality implies 
$$\|(T-T_{s,r})f_n\|_6\lesssim\Big(\int_{\Gamma\setminus\C(s,r)} |f_n|^2 d\sigma\Big)^{1/2}.$$
It follows that
\begin{align*}
\|T f_n\|_6&\leq \|T_{s,r}f_n\|_6+\|(T-T_{s,r})f_n\|_6\\
&\leq\|T_{s,r}f_n\|_6+C\Big(\int_{\Gamma\setminus\C(s,r)} |f_n|^2 d\sigma\Big)^{1/2}.
\end{align*}
Since $\{f_n\}$ concentrates at $\gamma(s)$, we have that
$$\lim_{n\rightarrow\infty}\|T f_n\|_6\leq \lim_{n\rightarrow\infty}\|T_{s,r}f_n\|_6+0$$ for every $r>0$. The result follows.
\end{proof}

We can now prove that condition \eqref{concentrates} in Proposition \ref{concentrationcompactness} cannot happen i.e. that the sequence  $\{F_n\}$ promised by that proposition cannot concentrate. Suppose it did concentrate at some point $\gamma(s)\in\Gamma$, and assume as before that $\|F_n\|_2\rightarrow 1$. By Corollary \ref{Fconcmincurv} we know that $\kappa(\gamma(s))=\lambda$. We may again assume that $\gamma(s)$ is not an endpoint of $\Gamma$. Then, by Lemma \ref{limlim} and Proposition \ref{limoc},

$${\mathbf C}[\Gamma]=\lim_{n\rightarrow\infty} \|\widehat{F_n\sigma}\|_6\leq\lim_{r\rightarrow 0^+}\lim_{n\rightarrow\infty}\|T_{s,r} F_n\|_6\leq\lim_{r\rightarrow 0^+}\|T_{s,r}\|=
{\mathbf C_F}[\lambda],$$
a contradiction to Proposition \ref{optimalconstants}. This concludes the proof of Theorem \ref{main}.

\section{Appendix 1:  $\| v_\epsilon\|_6^6$ is twice differentiable at $\epsilon=0$}
 
\noindent Recall that 
\begin{equation*}
v_\epsilon(x,t)=\int_\R (G_0+\epsilon\varphi)(u)e^{-i(x, t)\cdot (u,\frac{\lambda u^2}{2}+a\epsilon^2 u^4+\epsilon^{-2}\psi(\epsilon u))}\Big(1+\frac{\lambda^2}{2} \epsilon^2u^2\Big)\eta_I\Big(\frac{1}{\log\frac{1}{\epsilon}} u\Big)du,
\end{equation*}
where $0<\epsilon\leq\epsilon_0$, $\lambda=\min_\Gamma \kappa$, $G_0(u)=e^{-\lambda u^2/2}$, $\varphi(u)= c_\lambda u e^{-\lambda u^2/2}$, $a\in [\lambda^3/8,3\lambda^3/16)$, $\psi$ is a real-valued smooth function satisfying $\psi(y)=O(|y|^5)$ as $|y|\rightarrow 0$, and $\eta_I$ is a mollified version of the characteristic function of the interval $I$.

 The main goal of this appendix is to prove the following: 
\begin{proposition}\label{d2epsilon}
The function $\epsilon\mapsto \|v_\epsilon\|_{L^6(\R^2)}^6$ is twice differentiable at $\epsilon=0$, and 
\begin{multline}
\partial_\epsilon^2|_{\epsilon=0} \|v_\epsilon\|_6^6=  18\iint |G_1|^4|\varphi_1|^2 dxdt+12\Re\iint |G_1|^2\overline{G_1}^2\varphi_1^2 dxdt\label{w2here}\\
+3\pi^{3/2}\lambda^{-1/2}{\mathbf C_F}[\lambda]^6-8a\pi^{3/2}\lambda^{-7/2}{\mathbf C_F}[\lambda]^6.
\end{multline}
\end{proposition}

\noindent It is enough to show that 
\begin{equation}\label{intd2v}
\iint_{\R^2} \Big|\partial_\epsilon |v_\epsilon (x,t)|^6\Big| dxdt\leq C\;\textrm{ and }\;\iint_{\R^2} \Big|\partial_\epsilon^2 |v_\epsilon (x,t)|^6\Big| dxdt\leq C
\end{equation}
for a finite constant $C$ (independent of $\epsilon$) and every $0<\epsilon\leq\epsilon_0$.
 The  existence of $\partial_\epsilon^2|_{\epsilon=0}\| v_\epsilon\|_6^6$ will then follow from standard tools of analysis, and the formal computations from \S \ref{sec:The hard term} show that its value is given by formula \eqref{w2here}.

For $\epsilon<1$, set, as in the proof of Proposition \ref{propest},
$$G_\epsilon(u)=\Big(1+\frac{\lambda^2}{2} \epsilon^2u^2\Big)(G_0+\epsilon\varphi)(u).$$ 
We have the following expression for the oscillatory integral:
\begin{align*}
v_\epsilon(x,t)=
&\int_\R e^{-it\phi_\epsilon(u)}G_\epsilon(u)\eta_I\Big(\frac{1}{\log\frac{1}{\epsilon}} u\Big)du,
\end{align*}
where

\begin{equation}\label{expressionforphase}
\phi_\epsilon(u):=\frac{x}{t}u+\frac{\lambda u^2}{2}+a\epsilon^2 u^4+\epsilon^{-2}\psi(\epsilon u).
\end{equation}

In order to present the main estimates for $v_\epsilon$, it is convenient to perform a  decomposition of the $(x,t)$-plane which we now describe. Let ${\eta_0},{\eta}\in C_0^\infty(\R)$  be  even and smooth cut-off functions supported in $[-1,1]$ and $[-2,-1/2]\cup[1/2,2]$ respectively, with the  properties that $0\leq{{\eta_0}}\leq 1$, $0\leq{\eta}\leq 1$, and 
\begin{equation}\label{partition}
\eta_I\Big(\frac{1}{\log\frac{1}{\epsilon}} u\Big)={\eta_0}(u)+\sum_{k=1}^{K(\epsilon)} {\eta}(2^{-k+1}u)\;\;\;\textrm{ for every }u\in \R.
\end{equation}
This can be accomplished with $2^{K(\epsilon)}\asymp\log\frac{1}{\epsilon}$. We obtain, in particular, a smooth partition of unity in the interval $\frac{1}{2}\log\frac{1}{\epsilon}\cdot I$ subordinate to the dyadic regions $\mathcal{D}_k:=\{u\in\R: 2^{k-2}\leq|u|\leq 2^{k}\}$. 
This allows us to express $v_\epsilon$ as a sum of $\asymp K(\epsilon)$ integrals:

\begin{equation}\label{toomanyveps}
v_\epsilon(x,t)=v_{\epsilon,0}(x,t)+\sum_{k=1}^{K(\epsilon)} v_{\epsilon,k}(x,t),
\end{equation}
where
$$v_{\epsilon,0}(x,t):=\int_\R e^{-it\phi_\epsilon(u)}G_\epsilon(u){\eta_0}(u)du$$
and
\begin{equation}\label{vek}
v_{\epsilon,k}(x,t):=\int_\R e^{-it\phi_\epsilon(u)}G_\epsilon(u){\eta}(2^{-k+1} u)du
\end{equation}
\noindent for $k\in\{1,2,\ldots, K(\epsilon)\}$. Define:

\begin{align*}
&\mathfrak{C}_0:=\Big\{(x,t)\in\R^2: \Big|\frac{x}{t}\Big|\leq \lambda \Big\};\\
&\mathfrak{C}_k:=\Big\{(x,t)\in\R^2: \lambda 2^{k-2}\leq\Big|\frac{x}{t}\Big|\leq \lambda 2^{k}\Big\};\;\;(k\in\{1,2,\ldots, K(\epsilon)\})\\
&\mathfrak{C}_\infty:=\Big\{(x,t)\in\R^2: \Big|\frac{x}{t}\Big|\geq \lambda \log\frac{1}{\epsilon} \Big\}.
\end{align*}

\noindent This yields a decomposition of the $(x,t)$-plane as a union of cones
\begin{equation}\label{decR2}
 \R^2=\mathfrak{C}_0\cup\bigcup_{k=1}^{K(\epsilon)} \mathfrak{C}_k\cup \mathfrak{C}_\infty.
 \end{equation}
 which parallels \eqref{partition}.
 For $k\in\{2,3,\ldots, K(\epsilon)-2\}$, define the ``enlarged'' cones 
 
$$\mathfrak{C}_k^*:=\mathfrak{C}_k\cup \mathfrak{C}_{k\pm 1}\cup \mathfrak{C}_{k\pm 2}.$$
Additionally, let

\begin{align*}
&\mathfrak{C}_0^*:=\mathfrak{C}_0\cup \mathfrak{C}_1\cup \mathfrak{C}_2;\\
&\mathfrak{C}_1^*:=\mathfrak{C}_0\cup \mathfrak{C}_1\cup \mathfrak{C}_2\cup \mathfrak{C}_3;\\
&\mathfrak{C}_{K(\epsilon)-1}^*:=\mathfrak{C}_{K(\epsilon)-3}\cup \mathfrak{C}_{K(\epsilon)-2}\cup \mathfrak{C}_{K(\epsilon)-1}\cup \mathfrak{C}_{K(\epsilon)}\cup \mathfrak{C}_\infty;\\
&\mathfrak{C}_{K(\epsilon)}^*:=\mathfrak{C}_{K(\epsilon)-2}\cup \mathfrak{C}_{K(\epsilon)-1}\cup \mathfrak{C}_{K(\epsilon)}\cup \mathfrak{C}_\infty;\\
\end{align*}

The estimates in the following proposition are an expression of the {\em stationary phase principle}, which roughly states that the main contribution for an oscillatory expression like \eqref{vek} comes from the information concentrated on neighborhoods of the stationary points of its phase function.

\begin{proposition}\label{mainestimates}
For $k\in\{1,\ldots, K(\epsilon)\}$ and for every sufficiently small $\epsilon>0$, there exist $a_k\geq 0$ such that
\begin{displaymath}
|v_{\epsilon,k}(x,t)| \lesssim a_k\cdot \left\{ \begin{array}{ll}
\langle t \rangle^{-1/2} & \textrm{if $(x,t)\in \mathfrak{C}_k^*$},\\
\langle t \rangle^{-1} & \textrm{if $(x,t)\in \mathfrak{C}_0$},\\
\langle x \rangle^{-1} & \textrm{if $(x,t)\in \Big(\bigcup_{|j-k|>2}\mathfrak{C}_j\Big)\cup \mathfrak{C}_\infty$}.\\
\end{array} \right.
\end{displaymath}
and $\sum_{k} 2^ka_k<\infty$. If $k=0$, then
\begin{displaymath}
|v_{\epsilon,0}(x,t)| \lesssim  \left\{ \begin{array}{ll}
\langle t \rangle^{-1/2} & \textrm{if $(x,t)\in \mathfrak{C}_0^*$},\\
\langle x \rangle^{-1} & \textrm{otherwise}.\\
\end{array} \right.
\end{displaymath}
\end{proposition}

\begin{proof}

To make the notation less cumbersome, we will limit our discussion to the case when $k\in\{3,4,\ldots, K(\epsilon)-2\}$. The other cases follow in a similar way. 
\vspace{.5cm}

{\textbf{Case 1.} ${(x,t)\in \mathfrak{C}_k^*}$}

Without loss of generality we may assume that $(x,t)\in \mathfrak{C}_k^*$ is such that the phase $\phi_\epsilon$ has a critical point in the support of the cut-off function $\eta_I\Big(\frac{1}{\log\frac{1}{\epsilon}} \cdot\Big)$, for otherwise we could integrate \eqref{vek} by parts and obtain a better decay in $t$. This critical point is necessarily unique. In other words, there exists a unique $u_0\in 2\log\frac{1}{\epsilon}\cdot I$ such that
\begin{equation}\label{s_0}
\frac{d\phi_\epsilon}{du}(u_0)=\frac{x}{t}+\lambda u_0+4a\epsilon^2 u_0^3+\epsilon^{-1}\psi'(\epsilon u_0)=0.
\end{equation}
In general one cannot hope to solve this equation explicitly for $u_0$. However, since $|u_0|\lesssim \log\frac{1}{\epsilon}$, we have that 
$$u_0=-\frac{x}{\lambda t}+O_{a,\psi}(\epsilon).$$
In particular, since $(x,t)\in \mathfrak{C}_k^*$, we have that $|u_0|\asymp 2^k$.

Translating $u\mapsto u+u_0$ and defining $\widetilde{\phi}_\epsilon(u):=\phi_\epsilon(u+u_0)-\phi_\epsilon(u_0)$, we have that
\begin{equation*}
v_\epsilon(x,t)=e^{-it{\phi_\epsilon}(u_0)}\int_\R e^{-it\widetilde{\phi}_\epsilon(u)}G_\epsilon(u+u_0)\eta_I \Big(\frac{1}{\log\frac{1}{\epsilon}} (u+u_0)\Big)du.
\end{equation*}

\noindent It will suffice to get good estimates on  
\begin{equation}\label{yetanotherv}
\widetilde{v_\epsilon}(x,t):= e^{it{\phi_\epsilon}(u_0)} v_\epsilon(x,t)=\int_\R e^{-it\widetilde{\phi}_\epsilon(u)}G_\epsilon(u+u_0)\eta_I\Big(\frac{1}{\log\frac{1}{\epsilon}} (u+u_0)\Big)du.
\end{equation}

\noindent The new phase function $\widetilde{\phi_\epsilon}$ satisfies 
$$\widetilde{\phi_\epsilon}(0)=0=\frac{d\widetilde{\phi_\epsilon}}{du}(0)\;\textrm{ and }\;\frac{d^2\widetilde{\phi_\epsilon}}{du^2}(0)=\lambda+O(\epsilon).$$
In particular, the origin is its unique nondegenerate critical point.
This property is shared by the quadratic function $v\mapsto \frac{v^2}{2}$. Inspired by the proof\footnote{See, for instance, \cite[pp. 334$-$337]{S}.} of the usual method of stationary phase, we will change variables once again.

 Recalling definition \eqref{expressionforphase} and identity \eqref{s_0}, and using Taylor's formula, we have that

\begin{align*}
\frac{v^2}{2}:=&\widetilde{\phi_\epsilon}(u)=\phi_\epsilon(u+u_0)-\phi_\epsilon(u_0)\\
=&\frac{u^2}{2}\Big((\lambda+12 au_0^2\epsilon^2)+8au_0\epsilon^2 u+2a\epsilon^2 u^2\Big)+\epsilon^{-2}\psi(\epsilon(u+u_0))-\epsilon^{-2}\psi(\epsilon u_0)-u\epsilon^{-1}\psi'(\epsilon u_0)\\
=&\frac{u^2}{2}\Big((\lambda+12 au_0^2\epsilon^2)+8au_0\epsilon^2 u+2a\epsilon^2 u^2+\psi''(\epsilon(u_0+\theta u))\Big)
\end{align*}

\noindent for some $\theta\in (0,1)$. Taking square roots,
\begin{equation}\label{defphi}
v=u\Big((\lambda+12 au_0^2\epsilon^2)+8au_0\epsilon^2 u+2a\epsilon^2 u^2+\psi''(\epsilon(u_0+\theta u))\Big)^{1/2}=:\Phi_\epsilon(u).
\end{equation}
Take $\epsilon>0$ sufficiently small. Then $\Phi_\epsilon$ is a $C^\infty$ diffeomorphism from $2\log\frac{1}{\epsilon}\cdot I$ onto its image, whose inverse we denote by $\Psi_\epsilon:=\Phi_\epsilon^{-1}$. 
One can verify directly that $\frac{d\Phi_\epsilon}{du}(u)>0$ for every $u\in2\log\frac{1}{\epsilon}\cdot I$.
As a consequence,
\begin{equation}\label{onemorev}
\widetilde{v_\epsilon}(x,t)=\int_\R e^{-it\frac{v^2}{2}}G_\epsilon(\Psi_\epsilon(v)+u_0)\eta_I\Big(\frac{1}{\log\frac{1}{\epsilon}} (\Psi_\epsilon(v)+u_0)\Big)\frac{d\Psi_\epsilon}{dv}(v)dv.
\end{equation}

\noindent Rewriting \eqref{defphi} as 
\begin{align}
\Phi_\epsilon (u)
=&u\Big(\lambda+2a(2u_0^2+(2u_0+u)^2)\epsilon^2+\psi''(\epsilon(u_0+\theta u))\Big)^{1/2}\label{phiepsilon}
\end{align}
and recalling that $\lambda,a>0$ and $\psi''(\epsilon(u_0+\theta u))=O(\epsilon^3)$, we see that $\Phi_\epsilon$ is twice differentiable as a function of $\epsilon$ for sufficiently small $\epsilon$. The same holds for $\frac{d\Phi_\epsilon}{du}$, and using the chain rule one can draw similar conclusions about $\Psi_\epsilon$ and $\frac{d\Psi_\epsilon}{dv}$. 
\vspace{.2cm}

{\em
 In what follows, $C$ will be a finite non-zero constant that may change from line to line and depend on the parameters $\lambda, a$, the function $\psi$ and the interval $I$ but will always be independent of $\epsilon, x$ and $t$. This uniformity is crucial in our analysis.}
\vspace{.2cm}

We proceed to prove some uniform (in $\epsilon$) bounds for $\frac{d\Psi_\epsilon}{dv}$ and $\frac{d^2\Psi_\epsilon}{dv^2}$.
Start by observing that, for every sufficiently small $\epsilon>0$,
$$C^{-1}|u|\leq |\Phi_\epsilon(u)|\leq C|u|,\;\;\;\forall u\in2\log\frac{1}{\epsilon}\cdot I.$$

\noindent Since $\Phi_\epsilon\circ\Psi_\epsilon=id$ on $\Phi_\epsilon(2\log\frac{1}{\epsilon}\cdot I)\subseteq C\log\frac{1}{\epsilon}\cdot I$, we have the same uniform bounds for $\Psi_\epsilon$:
\begin{equation}\label{unif1}
C^{-1}|v|\leq |\Psi_\epsilon(v)|\leq C|v|,\;\;\;\forall v\in C\log\frac{1}{\epsilon}\cdot I.
\end{equation}

\noindent We also have the uniform bounds
\begin{equation}\label{unif2}
C^{-1}\leq \Big|\frac{d\Phi_\epsilon}{du}(u)\Big|\leq C,\;\;\;\forall u\in2\log\frac{1}{\epsilon}\cdot I.
\end{equation}

\noindent By the inverse function theorem,
\begin{equation}\label{dinv}
\frac{d\Psi_\epsilon}{dv}(v)=\frac{1}{\frac{d\Phi_\epsilon}{du}(\Psi_\epsilon(v))},
\end{equation}
and so \eqref{unif2} implies
\begin{equation}\label{unif3}
C^{-1}\leq \Big|\frac{d\Psi_\epsilon}{dv}(v)\Big|\leq C,\;\;\;\forall v\in C\log\frac{1}{\epsilon}\cdot I.
\end{equation}
Only the upper bound will be useful to us.
In a similar way one can conclude that 
\begin{equation}\label{unif7}
\Big|\frac{d^2\Psi_\epsilon}{dv^2}(v)\Big|\leq C,\;\;\;\forall v\in C\log\frac{1}{\epsilon}\cdot I.
\end{equation}

 We also need to estimate the Gaussian term $G_\epsilon$  appearing in  \eqref{onemorev} and some of its derivatives. The following claim, whose proof is straightforward and therefore omitted, provides good enough bounds:
\begin{claim}\label{uniformestimates2}
The following uniform estimates hold for every sufficiently small $\epsilon>0$, for every nonnegative integer $n$ and for every $u\in 2\log\frac{1}{\epsilon}\cdot I$: 

$$\Big|\frac{d^n}{d u^n}G_\epsilon(u)\Big|\lesssim_n \langle u\rangle^ne^{-\frac{\lambda u^2}{2}}.$$
\end{claim}
\vspace{.5cm}

Let us go back to \eqref{onemorev}. Introducing the cut-off functions $\eta$ and ${\eta_0}$ as before, we can expand
\begin{equation}\label{dectildev}
\widetilde{v}_\epsilon(x,t)=\widetilde{v}_{\epsilon,0}(x,t)+\sum_{k=1}^{K(\epsilon)} \widetilde{v}_{\epsilon,k}(x,t),
\end{equation}
where
$$\widetilde{v}_{\epsilon,0}(x,t):=\int_\R e^{-it\frac{v^2}{2}}G_\epsilon(\Psi_\epsilon(v)+u_0)\frac{d\Psi_\epsilon}{dv}(v){\eta_0}(\Psi_\epsilon(v)+u_0)dv$$
and
$$\widetilde{v}_{\epsilon,k}(x,t):=\int_\R e^{-it\frac{v^2}{2}}G_\epsilon(\Psi_\epsilon(v)+u_0)\frac{d\Psi_\epsilon}{dv}(v){\eta}(2^{-k+1} (\Psi_\epsilon(v)+u_0))dv$$
for $k\in\{1,\ldots, K(\epsilon)\}$. As before, $|\widetilde{v}_{\epsilon,k}|=|{v}_{\epsilon,k}|$ pointwise. 

 Define:
\begin{equation}\label{defbeta}
b_\epsilon(v):=G_\epsilon(\Psi_\epsilon(v)+u_0)\frac{d\Psi_\epsilon}{dv}(v)
\;\;\;\textrm{ and } \;\;\;
{\eta}_k(v):={\eta}(2^{-k+1} (\Psi_\epsilon(v)+u_0)).
\end{equation}
Notice that ${\eta}_k\in C_0^\infty(\R)$ is supported on 
$$\mathcal{E}_k:=\{v\in  C\log\frac{1}{\epsilon}\cdot I: 2^{k-2}\leq|\Psi_\epsilon(v)+u_0|\leq 2^{k}\},$$ 
and that $b_\epsilon$ is a Schwartz function on $\mathcal{E}_k$.
This justifies the use of  Plancherel's Theorem:
\begin{equation*}
\widetilde{v}_{\epsilon,k}(x,t)=\int_\R e^{-it\frac{v^2}{2}}b_\epsilon(v){\eta}_k(v)dv= (-2\pi i)^{1/2}t^{-1/2}\int_\R e^{ i t^{-1} \frac{\xi^2}{2}}\widehat{b_\epsilon{\eta}_k}(\xi)d\xi.
\end{equation*}

 Set $a_k:=\|\widehat{b_\epsilon{\eta}_k}\|_{L^1}$. We will be done analyzing Case 1 once we verify that the sequence $\{a_k\}_{k}$ decays rapidly enough to force the series $\sum_{k} 2^ka_k$ to converge.

We start by estimating the $L^2$ norm of the function  $b_\epsilon{\eta}_k$. Changing back to the original variable $u=\Psi_\epsilon(v)$ and using H\"{o}lder's inequality together with estimate \eqref{unif3} and Claim \ref{uniformestimates2} for $n=0$, one gets that
\begin{align*}
\|b_\epsilon{\eta}_k\|_{L^2}^2=&\int \Big|G_\epsilon(\Psi_\epsilon(v)+u_0)\frac{d\Psi_\epsilon}{dv}(v){\eta}(2^{-k+1} (\Psi_\epsilon(v)+u_0))\Big|^2dv\\
\leq&\int_{|u+u_0|\asymp 2^k}\Big|G_\epsilon(u+u_0){\eta}(2^{-k+1} (u+u_0))\Big|^2\Big|\frac{d\Psi_\epsilon}{dv}(\Phi_\epsilon(u))\Big|du
\lesssim 2^k e^{-\lambda4^{k-1}}.
\end{align*}

An  analogous argument, using estimate \eqref{unif7} instead,
yields 
$$\Big\|\frac{d}{dv}(b_\epsilon{\eta}_k)\Big\|_{L^2}^2\lesssim 2^{3k} e^{-\lambda4^{k-1}}.$$ 

Using Cauchy-Schwarz and Plancherel, we see that  these two estimates are enough for our purposes:
\begin{align*}
a_k=\|\widehat{b_\epsilon{\eta}_k}\|_{L^1}=&\int_{|\xi|\leq 1}|\widehat{b_\epsilon{\eta}_k}(\xi)|d\xi+\int_{|\xi|\geq 1}\frac{1}{|\xi|}\Big(|\xi||\widehat{b_\epsilon{\eta}_k}(\xi)|\Big)d\xi\\
\lesssim& \|\widehat{b_\epsilon{\eta}_k}\|_{L^2}+\Big(\int_{|\xi|\geq 1}{|\xi|^{-2}}d\xi\Big)^{1/2}\Big(\int_{|\xi|\geq 1} |\xi|^2|\widehat{b_\epsilon{\eta}_k}(\xi)|^2d\xi\Big)^{1/2}\\
\lesssim& \|{b_\epsilon{\eta}_k}\|_{L^2}+\Big\|\frac{d}{du}(b_\epsilon{\eta}_k)\Big\|_{L^2}
\lesssim 2^{3k/2} e^{-\lambda2^{2k-3}}.
\end{align*}
This concludes the analysis of Case 1.
\vspace{.5cm}

{\textbf{Case 2.} ${(x,t)\in \mathfrak{C}_0}$}

 The crucial observation is that, since $(x,t)\in \mathfrak{C}_0$ and $k>2$, the phase $\phi_\epsilon$ has no critical points in the support of $\eta(2^{-k+1}\cdot)$ i.e. the dyadic region $\mathcal{D}_k=\{u\in\R: 2^{k-2}\leq|u|\leq 2^{k}\}$. Indeed, since $|\frac{x}{t}|\leq\lambda$, we have that
\begin{equation}\label{dphilowerbd}
\Big|\frac{d\phi_\epsilon}{du}(u)\Big|=\Big|\frac{x}{t}+\lambda u+ O(\epsilon)\Big|\geq\frac{1}{2}\Big( |\lambda u|-\Big|\frac{x}{t}\Big|\Big)\geq\frac{\lambda}{2}(2^{k-2}-1)\geq\frac{\lambda}{2}
\end{equation}
if $\epsilon>0$ is chosen sufficiently small.

Integrating \eqref{vek} by parts, we get 
\begin{align*}
|v_{\epsilon,k}(x,t)|=&\frac{1}{|t|}\Big|\int_{\mathcal{D}_k}  e^{-it\phi_\epsilon(u)}\frac{d}{du}\Big(\frac{G_\epsilon(u)\eta(2^{-k+1}u)}{\frac{d\phi_\epsilon}{du}(u)}\Big)du\Big|\\
\lesssim&\frac{1}{|t|}\int_{\mathcal{D}_k} \Big|\frac{\frac{d}{du}(G_\epsilon (u)\eta(2^{-k+1}u))}{\frac{d\phi_\epsilon}{du}(u)}\Big|+\Big|\frac{G_\epsilon (u)\eta(2^{-k+1}u)\frac{d^2\phi_\epsilon}{du^2}(u)}{(\frac{d\phi_\epsilon}{du}(u))^2}\Big|du.
\end{align*}
H\"{o}lder's inequality implies
$$|v_{\epsilon,k}(x,t)|\lesssim \frac{2^k}{|t|}\Big(\Big\|\frac{\frac{d}{du}(G_\epsilon \eta(2^{-k+1}\cdot))}{\frac{d\phi_\epsilon}{du}}\Big\|_{L^\infty(\mathcal{D}_k)}+\Big\|\frac{G_\epsilon \eta(2^{-k+1}\cdot)\frac{d^2\phi_\epsilon}{du^2}}{(\frac{d\phi_\epsilon}{du})^2}\Big\|_{L^\infty(\mathcal{D}_k)}\Big),$$
and the desired estimate\footnote{One could repeat this argument $N$ times and obtain a bound 
$|v_{\epsilon,k}(x,t)|\lesssim_{k,N} \langle t\rangle^{-N}$ for  $(x,t)\in \mathfrak{C}_0$,
but this extra knowledge would be of no significance  to our analysis.} now follows from \eqref{dphilowerbd}, Claim \ref{uniformestimates2} and the fact that $\frac{d^2\phi_\epsilon}{du^2}$ is uniformly bounded on $\mathcal{D}_k$.  
\vspace{.5cm}

{\textbf{Case 3.} ${(x,t)\in \mathfrak{C}_j}\textrm{ for some }j\textrm{ such that } |k-j|>2, \textrm{ or } {(x,t)\in \mathfrak{C}_\infty}$}
The proof is identical to that of Case 2 and is therefore omitted.
\end{proof}
\vspace{.2cm}

Let us go back to the proof of Proposition \ref{d2epsilon}. 
Observe that the  estimates from Proposition \ref{mainestimates} readily imply the following special case of the $L^2\rightarrow L^6$ adjoint restriction inequality:
\begin{equation}\label{intv}
\iint_{\R^2}  |v_\epsilon (x,t)|^6 dxdt\leq C.
\end{equation}

\noindent Indeed, using the expansion \eqref{toomanyveps}, we have that
$$\iint_{\R^2}  |v_\epsilon (x,t)|^6 dxdt=\sum_{k_1,\ldots,k_6}\iint_{\R^2}{v}_{\epsilon,k_1}\overline{{v}_{\epsilon,k_2}}{v}_{\epsilon,k_3}\overline{{v}_{\epsilon,k_4}}{v}_{\epsilon,k_5}\overline{{v}_{\epsilon,k_6}}dxdt$$
where, for each $j\in\{1,\ldots,6\}$, the sum is taken over $k_j\in\{0,1,\ldots,K(\epsilon)\}$. For a fixed $(k_1,\ldots,k_6)$ the corresponding integral can be written as a sum of $\asymp K(\epsilon)$ integrals over the regions given by decomposition \eqref{decR2},
and using the bounds given by Proposition \ref{mainestimates} on each of these regions one readily obtains \eqref{intv}.
Note that
$$\iint_{\mathfrak{C}_k} \langle t\rangle^{-3}dxdt\asymp 2^k,$$
and so it is crucial to know that $\sum_k 2^ka_k<\infty$.

To prove \eqref{intd2v}, it is enough to control the following integrals:

\begin{align*}
I_0(\epsilon):=&\iint_{\R^2} |v_\epsilon (x,t)|^5 |\partial_\epsilon v_\epsilon(x,t)|  dxdt;\\
I_1(\epsilon):=&\iint_{\R^2} |v_\epsilon (x,t)|^5 |\partial_\epsilon^2 v_\epsilon(x,t)|  dxdt;\\
I_2(\epsilon):=&\iint_{\R^2} |v_\epsilon (x,t)|^4 |\partial_\epsilon v_\epsilon(x,t)|^2  dxdt.
\end{align*}

\noindent The reasoning just described to prove \eqref{intv} can be used to establish bounds for the integrals $I_0(\epsilon)$, $I_1(\epsilon)$ and $I_2(\epsilon)$ which are uniform in $\epsilon$,
as long as we have an analogue of Proposition \ref{mainestimates} for first and second derivatives.
As before, bounds for $\partial_\epsilon \widetilde{v}_\epsilon$ and $\partial_\epsilon^2 \widetilde{v}_\epsilon$ will suffice.
Let us focus on the more involved case of second derivatives: 

\begin{proposition}\label{mainestimates2}
For $k\in\{1,\ldots, K(\epsilon)\}$ and for every sufficiently small $\epsilon>0$, there exist $b_k\geq 0$ such that
\begin{displaymath}
|\partial_\epsilon^2 \widetilde{v}_{\epsilon,k}(x,t)| \lesssim b_k\cdot \left\{ \begin{array}{ll}
\langle t \rangle^{-1/2} & \textrm{if $(x,t)\in \mathfrak{C}_k^*$},\\
\langle t \rangle^{-1} & \textrm{if $(x,t)\in \mathfrak{C}_0$},\\
\langle x \rangle^{-1} & \textrm{if $(x,t)\in \Big(\bigcup_{|j-k|>2}\mathfrak{C}_j\Big)\cup \mathfrak{C}_\infty$}.\\
\end{array} \right.
\end{displaymath}
and $\sum_{k} 2^kb_k<\infty$. If $k=0$, then
\begin{displaymath}
|\partial_\epsilon^2 \widetilde{v}_{\epsilon,0}(x,t)| \lesssim  \left\{ \begin{array}{ll}
\langle t \rangle^{-1/2} & \textrm{if $(x,t)\in \mathfrak{C}_0^*$},\\
\langle x \rangle^{-1} & \textrm{otherwise}.\\
\end{array} \right.
\end{displaymath}
\end{proposition}

\noindent The proof follows the same steps of Proposition \ref{mainestimates}, with only one difference:

\begin{center}
{\em We need appropriate bounds for $\partial_\epsilon\Big(\frac{d^j\Psi_\epsilon}{dv^j}\Big)$ and $\partial_\epsilon^2\Big(\frac{d^j\Psi_\epsilon}{dv^j}\Big)$ for $j\in\{1,2\}$.}
\end{center}
Let us briefly outline how to accomplish this. Using \eqref{defphi}, we have that
$$\Phi_\epsilon (\Psi_\epsilon(v))=\Psi_\epsilon(v)\Big(\lambda+(12 au_0^2+8au_0 \Psi_\epsilon(v)+2a \Psi_\epsilon(v)^2)\epsilon^2+\psi''(\epsilon(u_0+\theta \Psi_\epsilon(v)))\Big)^{1/2}.$$
Differentiate both sides of the last identity with respect to $\epsilon$, the left hand side being obviously equal to 0. For the right hand side, we get two kinds of terms, depending on whether or not they contain a factor of the form $\partial_\epsilon \Psi_\epsilon$. Grouping together in one side of the equation all the terms which do contain such a factor, we can estimate:
\begin{equation}\label{unif4}
|\partial_\epsilon \Psi_\epsilon(v)|\leq C,\;\;\;\forall v\in C\log\frac{1}{\epsilon}\cdot I.
\end{equation}

\noindent It is also elementary to show that\begin{equation}\label{unifn}
\Big|\partial_\epsilon \Big(\frac{d\Phi_\epsilon}{du}\Big)(u)\Big|\leq C,\;\;\;\forall u\in 2\log\frac{1}{\epsilon}\cdot I.
\end{equation} 

\noindent Differentiating both sides of \eqref{dinv} with respect to $\epsilon$, we obtain
$$\partial_\epsilon \Big(\frac{d\Psi_\epsilon}{dv}\Big)(v)=-\frac{\partial_\epsilon\Big(\frac{d\Phi_\epsilon}{du}(\Psi_\epsilon(v))\Big)}{\Big(\frac{d\Phi_\epsilon}{du}(\Psi_\epsilon(v))\Big)^2}.$$
\noindent Using estimates \eqref{unif2} and \eqref{unif4}, we similarly conclude that
\begin{equation}\label{unif5}
\Big|\partial_\epsilon \Big(\frac{d\Psi_\epsilon}{dv}\Big)(v)\Big|\leq C,\;\;\;\forall v\in C\log\frac{1}{\epsilon}\cdot I.
\end{equation}
Repeating this whole procedure once again, we conclude in an analogous way that
\begin{equation}\label{unif6}
\Big|\partial_\epsilon^2 \Big(\frac{d\Psi_\epsilon}{dv}\Big)(v)\Big|\leq Cv^2,\;\;\;\forall v\in C\log\frac{1}{\epsilon}\cdot I.
\end{equation}
The terms $\partial_\epsilon \Big(\frac{d^2\Psi_\epsilon}{dv^2}\Big)$ and $\partial_\epsilon^2 \Big(\frac{d^2\Psi_\epsilon}{dv^2}\Big)$ can be dealt with in a similar way. Recalling what we already know from \eqref{unif3} and \eqref{unif7}, we arrive at the following lemma:

\begin{lemma}\label{unifdpsi}
The following estimates hold for $j\in\{1,2\}$, for every sufficiently small $\epsilon>0$ and for every $v\in C\log\frac{1}{\epsilon}\cdot I$:
\begin{itemize}
\item[(i)]
$\Big|\frac{d^j\Psi_\epsilon}{dv^j}(v)\Big|\leq C$;
\item[(ii)]
$\Big|\partial_\epsilon \Big(\frac{d^j\Psi_\epsilon}{dv^j}\Big)(v)\Big|\leq C$;
\item[(iii)]
$\Big|\partial_\epsilon^2 \Big(\frac{d^j\Psi_\epsilon}{dv^j}\Big)(v)\Big|\leq Cv^2$.
\end{itemize}
\end{lemma}

\noindent Lemma \ref{unifdpsi} can be used together with the estimates from Claim \ref{uniformestimates2} to prove Proposition \ref{mainestimates2}. We omit the details.

\section{Appendix 2: Two explicit calculations}

\noindent{\textbf{Claim \ref{c1}.}}
\begin{equation}\label{c1second}
3\lambda^2\Re\iint_{\R^2} |G_1(x,t)|^4\overline{G_1(x,t)}G_2(x,t) dxdt=\frac{3}{2}\pi^{3/2}\lambda^{-1/2}{\mathbf C_F}[\lambda]^6.
\end{equation}

\begin{proof}
As first observations, note that 
\begin{equation}\label{first}
\overline{G_1(x,t)}G_2(x,t)=\Big(\lambda^{-1}(1+it)^{-1}-\lambda^{-2}x^2(1+it)^{-2}\Big) |G_1(x,t)|^2
\end{equation}
and 
$$|G_1(x,t)|^2=\frac{2\pi}{\lambda}(1+t^2)^{-1/2}e^{-\frac{x^2}{\lambda(1+t^2)}}.$$
It follows that the left hand side in \eqref{c1second} equals
$$3\lambda^2\Big(\frac{2\pi}{\lambda}\Big)^3\Re\iint\Big(\lambda^{-1}(1+it)^{-1}-\lambda^{-2} x^2(1+it)^{-2}\Big)(1+t^2)^{-3/2}e^{-\frac{3x^2}{\lambda (1+t^2)}}dxdt.$$
Write $(1+it)^{-1}=(1+t^2)^{-1}(1-it)$, $(1+it)^{-2}=(1+t^2)^{-2}(1-it)^2$, and change variables $y=\frac{x}{(1+t^2)^{1/2}}$ to compute
\begin{align*}
{\mathbf I}:=\Re\iint(1+t^2)^{-1}(1-it)(1+t^2)^{-3/2}&e^{-\frac{3x^2}{\lambda (1+t^2)}}dxdt=\iint(1+t^2)^{-5/2}e^{-\frac{3x^2}{\lambda (1+t^2)}}dxdt\\
=&\Big(\int_{-\infty}^\infty\frac{1}{(1+t^2)^{2}}dt\Big)\cdot\Big(\int_{-\infty}^\infty e^{-\frac{3y^2}{\lambda}}dy\Big)\\
=&\frac{\pi}{2}\cdot\Big(\frac{\pi}{3/\lambda}\Big)^{1/2}=\frac{\pi^{3/2}}{2\sqrt{3}} \lambda^{1/2}
\end{align*}
and
\begin{align*}
{\mathbf{II}}:=&\Re\iint x^2(1+t^2)^{-2}(1-it)^2(1+t^2)^{-3/2}e^{-\frac{3x^2}{\lambda (1+t^2)}}dxdt\\
=&\iint (1+t^2)^{-7/2}(1-t^2)x^2e^{-\frac{3x^2}{\lambda (1+t^2)}}dxdt\\
=&\Big(\int_{-\infty}^\infty\frac{1-t^2}{(1+t^2)^{2}}dt\Big)\Big(\int_{-\infty}^\infty y^2e^{-\frac{3y^2}{\lambda}}dy\Big)=0.\\
\end{align*}

\noindent All in all we have that
\begin{align*}
3 \lambda^2\Re\iint |G_1(x,t)|^4\overline{G_1(x,t)}G_2(x,t) &dxdt=3 \lambda^2\Big(\frac{2\pi}{\lambda}\Big)^3(\lambda^{-1} {\mathbf I}- \lambda^{-2}{\mathbf{II}})=3\lambda^{-2}(2\pi)^3 {\mathbf I}+0\\
=&\frac{3}{2}\pi^{3/2}\lambda^{-1/2}\frac{(2\pi)^3}{\sqrt{3}}\lambda^{-1}=\frac{3}{2}\pi^{3/2}\lambda^{-1/2}{\mathbf C_F}[\lambda]^6.\qedhere
\end{align*}
\end{proof}

\noindent{\textbf{Claim \ref{c2}.}}
\begin{equation}\label{c2second}
-6a\iint_{\R^2} \Re\Big\{it|G_1(x,t)|^4\overline{G_1(x,t)}G_3(x,t) \Big\}dxdt=-4a\pi^{3/2}\lambda^{-7/2}{\mathbf C_F}[\lambda]^6.
\end{equation}

\begin{proof}
As with \eqref{first}, we have that

$$\overline{G_1(x,t)}G_3(x,t)=\Big(3\lambda^{-2}(1+it)^{-2}-6\lambda^{-3}x^2(1+it)^{-3}+\lambda^{-4}x^4(1+it)^{-4}\Big)  |G_1(x,t)|^2,$$
and so the left hand side in \eqref{c2second} equals

\begin{multline*}
-6a\Big(\frac{2\pi}{\lambda}\Big)^3\iint \Re\Big\{it \Big(3 \lambda^{-2}(1+it)^{-2}-6 \lambda ^{-3}x^2(1+it)^{-3}+ \lambda^{-4}x^4(1+it)^{-4}\Big)\cdot\\
\cdot(1+t^2)^{-3/2}e^{-\frac{3x^2}{\lambda(1+t^2)}}\Big\}dxdt.
\end{multline*}

Note that $\Re \{i(1+it)^{-2}\}=2t(1+t^2)^{-2}$, $\Re \{i(1+it)^{-3}\}=(3t-t^3)(1+t^2)^{-3}$ and $\Re \{i(1+it)^{-4}\}=(4t-4t^3)(1+t^2)^{-4}$. Change variables $y=\frac{x}{(1+t^2)^{1/2}}$ to compute

\begin{align*}
{\mathbf I}:=\iint 2t^2(1+t^2)^{-2}(1+t^2)^{-3/2}&e^{-\frac{3x^2}{\lambda (1+t^2)}}dxdt=\iint 2t^2(1+t^2)^{-7/2}e^{-\frac{3x^2}{\lambda (1+t^2)}}dxdt\\
=&\Big(\int_{-\infty}^\infty\frac{2t^2}{(1+t^2)^{3}}dt\Big)\cdot\Big(\int_{-\infty}^\infty e^{-\frac{3y^2}{\lambda}}dy\Big)\\
=&\frac{\pi}{4}\cdot\Big(\frac{\pi}{3/\lambda}\Big)^{1/2}=\frac{\pi^{3/2}}{4\sqrt{3}} \lambda^{1/2},
\end{align*}

\begin{align*}
{\mathbf{II}}:=&\iint x^2 (3t^2-t^4)(1+t^2)^{-3}(1+t^2)^{-3/2}e^{-\frac{3x^2}{\lambda (1+t^2)}}dxdt\\
=&\iint (3t^2-t^4)(1+t^2)^{-9/2}x^2 e^{-\frac{3x^2}{\lambda (1+t^2)}}dxdt\\
=& \Big(\int_{-\infty}^\infty\frac{3t^2-t^4}{(1+t^2)^{3}}dt\Big)\Big(\int_{-\infty}^\infty y^2e^{-\frac{3y^2}{\lambda}}dy\Big)=0
\end{align*}
and

\begin{align*}
{\mathbf{III}}:=&\iint x^4 (4t^2-4t^4)(1+t^2)^{-4}(1+t^2)^{-3/2}e^{-\frac{3x^2}{\lambda (1+t^2)}}dxdt\\
=&\iint(4t^2-4t^4)(1+t^2)^{-11/2}x^4e^{-\frac{3x^2}{\lambda (1+t^2)}}dxdt\\
=&\Big(\int_{-\infty}^\infty\frac{4t^2-4t^4}{(1+t^2)^{3}}dt\Big)\cdot\Big(\int_{-\infty}^\infty y^4e^{-\frac{3y^2}{\lambda}}dy\Big)\\
=& -\pi\cdot\Big(\frac{1}{12}\Big(\frac{\pi}{3/\lambda}\Big)^{1/2}\Big)=-\frac{\pi^{3/2}}{12\sqrt{3}} \lambda^{5/2}.
\end{align*}

\noindent All in all we have that
\begin{align*}
-6a\iint \Re\Big\{it|G_1(x,t)|^4\overline{G_1(x,t)}G_3(x,t) \Big\}dxdt&=-6a\Big(\frac{2\pi}{\lambda}\Big)^3(3\lambda^{-2}{\mathbf I}-6\lambda^{-3}{\mathbf{II}}+\lambda^{-4}{\mathbf{III}})\\
&=-a(18\lambda^{-2}{\mathbf I}+0+6\lambda^{-4}{\mathbf{III}})\Big(\frac{2\pi}{\lambda}\Big)^3\\
&=-\Big(\frac{18}{4}-\frac{6}{12}\Big)a\pi^{3/2}\lambda^{-7/2}\frac{(2\pi)^3}{\sqrt{3}}\lambda^{-1}\\
&=-4a\pi^{3/2}\lambda^{-7/2}{\mathbf C_F}[\lambda]^6,
\end{align*}
as claimed.
\end{proof}

\end{document}